\documentclass[12pt]{amsart}
\usepackage{amssymb, amsfonts, amsthm} %, verbatim}
\usepackage{graphicx}
\usepackage{color}

\newtheorem{theorem}[equation]{Theorem}
\newtheorem{lemma}[equation]{Lemma}
\newtheorem{corollary}[equation]{Corollary}
\newtheorem{proposition}[equation]{Proposition}
\newtheorem*{theorem*}{Theorem}

\numberwithin{equation}{section}

\theoremstyle{definition}
\newtheorem{definition}[equation]{Definition}

\newtheorem{example}[equation]{Example}

\newtheorem{problem}[equation]{Problem}

\newtheorem{openproblem}[equation]{Open problem}

\newcommand{\be}{\begin{equation}}
\newcommand{\ee}{\end{equation}}

\newcommand{\Rn}{{\mathbb R}^n}
\newcommand{\Rnbar}{\overline{{\mathbb R}}^n}
\newcommand{\R}{{\mathbb R}}
\newcommand{\Rt}{{\mathbb R}^2}

\newcommand{\psubset}{\varsubsetneq}

\newcommand{\ds}{\displaystyle}

\newcommand{\comment}[1]{}

\newcounter{minutes}
\divide\time by 60
\newcounter{hours}
\multiply\time by 60 \addtocounter{minutes}{-\time}
\begin{document}

\title[On the stability of $\varphi$-uniform domains]
{On the stability of $\varphi$-uniform domains}
\author[R. Kl\'en]{R. Kl\'en}
\address{Department of Mathematics and Statistics, University of Turku,
FIN-20014 Turku, Finland}
\email{riku.klen@utu.fi}
\author[Y. Li]{Y. Li}
\address{Department of Mathematics,
Hunan Normal University, Changsha,  Hunan 410081, People's Republic
of China} \email{yaxiangli@163.com}
\author[S.K. Sahoo]{S.K. Sahoo}
\address{Department of Mathematics, Indian Institute of Technology Indore, Indore 452 017, India}
\email{swadesh@iitm.ac.in}
\author[M. Vuorinen]{M. Vuorinen}
\address{Department of Mathematics and Statistics, University of Turku,
FIN-20014 Turku, Finland}
\email{vuorinen@utu.fi}

%\thanks{The first author was supported in part by the Academy of Finland,
%the Research Council of Norway, Project 160192/V30, and the
%National Board for Higher Mathematics (DAE, India).}

% \date{Received date / Revised version date}


\begin{abstract}
We study two metrics, the quasihyperbolic metric and the distance ratio metric of
a subdomain $G \subset {\mathbb R}^n$. 
%and prove several inequalities between them.
In the sequel, we investigate a class of domains, so called $\varphi$-uniform domains, defined by the property
that these two metrics are comparable with respect to a homeomorphism $\varphi$
from $[0,\infty)$ to itself. Finally, we discuss a number of stability
properties of $\varphi$-uniform domains. In particular, we show that the class of
$\varphi$-uniform domains is stable in the sense that removal of a geometric
sequence of points from a $\varphi$-uniform domain yields a  $\varphi_1$-uniform domain.
\end{abstract}

\maketitle

\def\thefootnote{}
\footnotetext{ \texttt{\tiny File:~\jobname .tex,
          printed: \number\year-\number\month-\number\day,
          \thehours.\ifnum\theminutes<10{0}\fi\theminutes }
} \makeatletter\def\thefootnote{\@arabic\c@footnote}\makeatother
%===============================================================================

%%%%%%%%%%%%%%%%%%%%%%%%%%%%%%%%%
%%%%%%%%%%% SECTION 1 %%%%%%%%%%%
%%%%%%%%%%%%%%%%%%%%%%%%%%%%%%%%%

\section{Introduction}
\label{intro}

For a subdomain $G \psubset {\mathbb R}^n$ and
$x,y\in G$ the {\em distance ratio
metric $j_G$} is defined by
$$ j_G(x,y)=\log\left(1+\frac{|x-y|}{\min\{\delta_G(x),\delta_G(y)\}}\right)\,,
$$ where $\delta_G(x)$ denotes the Euclidean distance from $x$ to $\partial G$. Sometimes
we abbreviate $\delta_G$ by writing just $\delta\, .$
The above form of the $j_G$ metric, introduced in \cite{Vu2}, is obtained by a slight modification
of a metric that was studied in \cite{GO,GP}.
% In a slightly different form of this metric was studied in
% \cite{GP,GO} and in its above form in \cite{Vu2}.
The {\em quasihyperbolic metric} of $G$ is defined by
the quasihyperbolic length minimizing property
$$k_G(x,y)=\inf_{\gamma\in \Gamma(x,y)} \ell_k(\gamma),
\quad \ell_k(\gamma) =\int_\gamma \frac{|dz|}{\delta_G(z)}\,,
$$
where $\Gamma(x,y)$ represents the family of all rectifiable paths
joining $x$ and $y$ in $G$, and $\ell_k(\gamma)$ is the quasihyperbolic length of
$\gamma$ (cf. \cite{GP}).
For a given pair of points $x,y\in G,$ the infimum is always
attained \cite{GO}, i.e., there always exists a quasihyperbolic
geodesic $J_G[x,y]$ which minimizes the above integral,
$k_G(x,y)=\ell_k(J_G[x,y])\,$ and furthermore with the property that
the distance is additive on the geodesic: $k_G(x,y)=$
$k_G(x,z)+k_G(z,y) $ for all $z\in J_G[x,y]\,.$ If the domain $G$ is
emphasized we call $J_G[x,y]$ a $k_G$-geodesic. In this paper, sometimes we also use
the terminology {\em distance} for the term {\em metric}.

The following well-known properties of the above two metrics 
are useful in this paper.

\begin{enumerate}
\item[(i)] For $x,y\in G\psubset\Rn$, we have $k_G(x,y)\ge j_G(x,y)$ \cite{GP};
\item[(ii)] Monotonicity property: if $G_1$ and $G_2$ are domains, 
with $ G_2\subset G_1\psubset\Rn\,$, then for all
$x,y\in G_2$ we have $k_{G_1}(x,y)\le k_{G_2}(x,y).$ It is obvious
that this property holds for the distance ratio metric $j_G$ as well.
\end{enumerate}

%Let $(X_j,d_j)$, $j=1,2$, be metric spaces. A function $f:\, (X_1,d_1)\to
%(X_2,d_2)$ is said to be {\em uniformly continuous} if there exists a
%function, called a {\em modulus of continuity} $\omega:\,[0,r_1)\to [0,r_2)$
%such that $\omega(0)=0$ and $\omega(t)\to 0$, as $t\to 0$, and for
%all $x,y\in X_1$ with $d_1(x,y)< r_1$ we have $d_2(f(x),f(y))<
%\omega(d_1(x,y))<r_2$. In this case we also say that $f:\,(X_1,d_1)\to
%(X_2,d_2)$ is {\em $\omega$-uniformly continuous}. To simplify the
%matters, we always assume that $\omega:\,[0,r_1)\to [0,r_2)$ is an
%increasing homeomorphism. For instance,
%we see from \cite[7.17 (3), p. 378]{AVV} that the identity mapping
%$id :( {\mathbb{R}}^n, |\cdot|) \to ({\mathbb{R}}^n,q)$
%has the sharp modulus of continuity $\omega(t) = t/(1+(t/2)^2)$ for $t\in [0,2)$,
%where $q$ is denoted by the chordal distance in $\Rnbar$.
In 1979, Martio and Sarvas introduced the class of uniform domains \cite{MS79}.

%($\star$ {\bf Remark: Here I add the definition of uniform domain.})

\begin{definition}\label{unif} A domain $D$ in $\Rn$  is said to be  {\it $c$-uniform}  if there
exists a constant $c$ with the property that each pair of points
$z_{1},z_{2}$ in $D$ can be joined by a rectifiable arc $\gamma$ in
$ D$ satisfying (cf. \cite{MS79, Va91})
\begin{enumerate}
\item\label{eq-1}
$\ds\min_{j=1,2}\ell (\gamma [z_j, z])\leq c\, \delta_D(z)$ for all $z\in \gamma$, and

\item\label{eq-2}
$\ell(\gamma)\leq c\,|z_{1}-z_{2}|$,
\end{enumerate}
where $\ell(\gamma)$ denotes the arclength of $\gamma$,
$\gamma[z_{j},z]$ the part of $\gamma$ between $z_{j}$ and $z$.
Also, we say that $\gamma$ is a {\it uniform arc}.
A domain is said to be {\em uniform} if it is $c$-uniform for some constant $c>0$.
\end{definition}

In the same year,
Gehring and Osgood \cite{GO} characterized uniform domains in terms
of an upper bound for the quasihyperbolic metric as follows: a
domain $G$ is {\em uniform} if and only if there exists a constant $C\ge
1$ such that
\begin{equation}\label{C-uniform}
k_G(x,y)\le C j_G(x,y)
\end{equation}
for all $x,y\in G$. As a matter of fact, the above inequality
appeared in \cite{GO} in a form with an additive constant on the
right hand side: it was shown by Vuorinen \cite[2.50]{Vu2} that the
additive constant can be chosen to be $0$.  This observation leads to
the definition of $\varphi$-uniform domains  introduced in \cite{Vu2}.

\begin{definition}\label{psi-unif}Let $\varphi:\, [0,\infty)\to [0,\infty)$ be a homeomorphism.
A domain $G\psubset\Rn$ is said to be {\em $\varphi$-uniform} if
% \begin{equation} %\label{fiiunif-def}
$$k_G(x,y)\le \varphi(|x-y|/\min\{\delta(x),\delta(y)\})
$$
% \end{equation}
for all $x,y\in G$.\end{definition}

In the sequel, V\"ais\"al\"a has also investigated this
class of domains \cite{Va91} (see also \cite{Va98} and references therein).
He also pointed out that these two classes of domains are same provided
$\varphi$ is a slow function. We make sure that, in this paper, we use the terminology
{\em $c$-uniform} for constants $c$, and frequently use {\em $\psi$-uniform, $\vartheta$-uniform}
and {\em $\varphi$-uniform} for functions $\psi$, $\vartheta$, $\varphi$.

In Section~2, we introduce notation and preliminary results that we need in the latter sections. 
The structure of the rest of the sections covers mainly on $\varphi$-uniform domains.

In Section~3, we construct several examples of $\varphi$-uniform domains and
compare with their complementary domains and with quasiconvex domains.
We also prove that the image domain of a $\varphi$-uniform domain under
bilipschitz mappings of $\Rn$ is $\psi$-uniform, where $\psi$ is
depending on $\varphi$ and the bilipschitz constant.

In Section~4, we present our main results (e.g. see Theorems \ref{phiunif-E} and \ref{psiunif4})
on $\varphi$-uniform domains in the following form:

\begin{theorem*}
Let $G$ be a $\varphi$-uniform domain in $\Rn$. Let $B$ be a ball with $2B\subset G$. Suppose
that $E$ is a compact subset of $B$ such that $\Rn\setminus E$ is $\psi$-uniform. Then
$G\setminus E$ is $\vartheta$-uniform, where $\vartheta$ depends only on $\varphi$ and $\psi$.
\end{theorem*}

Note that one of our proofs involves a control function of a fixed parameter on which $\vartheta$ also depends.
Idea behind this is to obtain various other stability properties of $\varphi$-uniform domains
to use as preparatory results to prove the main theorems in the above type.
% geometric properties of $\varphi$-uniform domains.
In particular, it is shown that
the class of $\varphi$-uniform domains is stable in the sense that removal of
a geometric sequence of points from a $\varphi$-uniform domain leads to
a $\varphi_1$-uniform domain.

\section{Notation and Preliminary results}
\label{Prelim}

 We shall now specify some necessary notation,
definitions and facts that we frequently use in this paper. The
standard unit vectors in the Euclidean $n$-space $\Rn$ ($n\ge 2$)
are represented by $e_1,e_2,\ldots,e_n$. We write $x\in\Rn$ as a
vector $(x_1,x_2,\ldots,x_n)$.
% The Euclidean distance (Euclidean norm)
% of $x\in \Rn$ is denoted by $|x|$.
The Euclidean line segment joining points $x$ and $y$ is denoted by $[x,y]$.
For $x,y,z\in \Rn$, the smallest angle at $y$ between the vectors $x-y$ and $z-y$
is denoted by $\measuredangle (x,y,z)$.
The one point
compactification of $\Rn$ (so-called the M\"obius $n$-space) is
defined by $\Rnbar=\Rn\cup \{\infty\}$. We denote by $B^n(x,r)$ and
$S^{n-1}(x,r)$, the Euclidean ball and sphere with radius $r$
centered at $x$ respectively. We set $B^n(r):=B^n(0,r)$ and
$S^{n-1}(r):=S^{n-1}(0,r)$. Let $G$ be a domain (open connected
non-empty set) in $\Rn$. The boundary, closure and diameter of $G$
are denoted by $\partial G$, $\overline{G}$
and ${\rm diam}\,G$
respectively.  In what follows, all
paths $\gamma\subset G$ are required to be rectifiable, i.e.
$\ell(\gamma)<\infty$  where $\ell(\gamma)$
stands for the Euclidean length of $\gamma$. Given $x,y\in G$,
$\Gamma(x,y)$ stands for the collection of all rectifiable paths
$\gamma\subset G$ joining $x$ and $y\,.$

% Some basic properties of the quasihyperbolic and hyperbolic
% metrics will be frequently used (cf. \cite[Sections 2 and 3]{Vu})
% in this paper.
% First we recall the monotonicity property of the quasihyperbolic metric with respect to domains: if
% $G_1$ and $G_2$ are domains, with $ G_2\subset G_1\psubset\Rn\,$, then for all
% % $x,y\in G_2$ we have $k_{G_1}(x,y)\le k_{G_2}(x,y)\,.$ It is obvious
% that this property holds for the distance ratio metric $j_G$ as well.
% 
% 
% The hyperbolic metrics $\rho_{B^n}$ of the unit ball and
% $\rho_{{\mathbb H}^n}$ of the upper half space
% ${\mathbb H}^n = \{x \in {\mathbb R}^n:\, x_n >0 \}$
% are defined in terms of the hyperbolic length minimizing property,
% in the same way as the quasihyperbolic metric (see \cite[Section 2]{Vu}). The
% hyperbolic density functions are $2/(1-|x|^2)$ and $1/x_n$ for the unit ball
% and the upper half space, respectively. This leads to the observations that
% $k_{{\mathbb H}^n}= \rho_{{\mathbb H}^n}$ and
% $$\rho_{B^n}(x,y)/2 \le k_{B^n}(x,y)\le \rho_{B^n}(x,y)$$
% for all $x,y\in B^n\,.$
% For the case of $B^n$, we make use of an explicit formula
%  \cite[(2.18)]{Vu} to the effect that for
% $x,y\in B^n$
% \begin{equation} \label{rhodef}
% \sinh \frac{\rho_{B^n}(x,y)}2=\frac{|x-y|}{t}\, ,
% t=\sqrt{(1-|x|^2)(1-|y|^2)}\,.
% \end{equation}

We now formulate some basic results on quasihyperbolic distances
which are indeed used latter in Section \ref{Sec3}.
The following lemma is established in \cite{Vu2}.
\begin{lemma} \label{Vu2-lem} %\cite[Lemma~2.53]{Vu2}
Define
$$a(\theta)=1+(2/\theta)+\pi/\Big(2\log\frac{2+2\theta}{2+\theta}\Big),\quad \mbox{ for $\theta\in(0,1)$.}
$$
Let $G\psubset\Rn$ be a domain.
If $x,y,z\in G$ with $x,y\in G\setminus B^n(z,\theta \delta_G(z))$, then
$$k_{G\setminus\{z\}}(x,y)\le a(\theta)\,k_G(x,y)\,.
$$
\end{lemma}
It is seen from Lemma~\ref{Vu2-lem} that $a(\theta)$ is well-defined for $\theta=1$.
However, the method of the proof does not give any guarantee to obtain the same value of
$a(\theta)$ when $\theta=1$.

%This is a very simple result from my PhD thesis.
\begin{lemma}\label{complementofB}
If $r > 0$ and $x,y \in G = \Rn \setminus \overline{B}^n(r)$ with $|x|=|y|$, then
$$ k_G(x,y) \le \frac{|x|}{|x|-r}k_{\Rn \setminus \{ 0 \}}(x,y) \le \frac{|x-y|\,\pi}{2(|x|-r)}\,.
$$
\end{lemma}
\begin{proof}
The first inequality follows from \cite[Theorem~5.20]{Klen}. For the second inequality
we see that if $\theta=\measuredangle(x,0,y)$, then we have the identity
$$|x-y|^2=|x|^2+|y|^2-2|x|\,|y|\cos \theta\,.
$$
Since $|x|=|y|$ it follows that $\sin(\theta/2)=|x-y|/(2|x|)$. For $0\le \theta\le \pi$,
the well-known inequality $\theta\le \pi\sin(\theta/2)$ gives that $\theta\le \pi|x-y|/(2|x|)$.
Since $k_{\Rn \setminus \{ 0 \}}(x,y)\le \theta$, when $|x|=|y|$, we conclude the second inequality.
\end{proof}

%{\bf The following two lemmas are generalizations of Lemma \ref{Vu2-lem}.}
%$\star$ Remark: The modification until the end of the lemma  below.

\begin{lemma}\label{Lem1} For $r\in[\frac{1}{4},1)$, we define
$a(r)=\Big(4\Big(\frac{r+2}{4r-1}\Big)+\frac{2}{\log\frac{2+2r}{2+r}}\Big).$
Let $D$ be a proper subdomain of $\Rn$. If $x,y,z\in D$ with $y,z\in D\setminus B^n(x,r \delta_D(x)),$
$E=\{x\}\cup\{x_k\}_{k=1}^{\infty}$
where  $\{x_k\}_{k=1}^{\infty}\in B^n(x, \delta_D(x))$ is a sequence of points satisfying
 $x_k\in[x,x_{k-1})$ and $|x-x_k|=\frac{1}{2^{k+2}}\delta_D(x),$
then $$k_{D\setminus E}(y,z)\leq a(r)k_D(y,z).$$
\end{lemma}
\begin{proof}  Let $D_1=D\setminus  E ,$ and $\delta_D$ denote the Euclidean distance to the boundary of $D$.
Observe first that if $r\in(\frac{1}{4},1)$ and $w\in D \setminus B^n(x,\frac{r}{2}\delta_D(x))$,
then
 \begin{equation}\label{eq1'}\delta_D(w)\leq 4\Big(\frac{r+2}{4r-1}\Big)\delta_{D_1}(w),
\end{equation}
where the inequality holds because of the following.
If $\delta_D(w)=\delta_{D_1}(w),$ then \eqref{eq1'} holds trivially.
If $\delta_D(w)>\delta_{D_1}(w),$ then there exists some point $p\in E$ such that
$$\delta_{D_1}(w)=|w-p|\geq |w-x|-|x-p|\geq\Big(\frac{r}{2}-\frac{1}{8}\Big)\delta_D(x)\,.
$$
Hence
$$\delta_D(w)\leq \delta_D(x)+|x-p|+|p-w|\leq4\Big(\frac{r+2}{4r-1}\Big)\delta_{D_1}(w)\, .
$$

Let $\gamma$ be a quasihyperbolic geodesic joining $z$ and $y$ in $D$ and let
$U=\gamma\cap\overline{B}^n(x,\frac{r}{2}\delta_D(x)).$ We consider
two cases.

{\em Case I:}\;\;  $U\neq\emptyset$.

Let $z'$ be the first point in $U$ when we traverse along $\gamma$ from $z$ to $y$. The point $y'$ in $U$ is the first point
when we traverse from $y$ to $z$. Let $T$ be a $2$-dimensional linear subspace of $E$ containing $x$,$y'$ and $z'$. Then $y'$
and $z'$ divide the circle $T\cap S^{n-1}(x,\frac{r}{2}\delta_D(x))$ into two arcs, denote the shorter arc
(which may be a semicircle) by  $\alpha$.
Then  \eqref{eq1'} yields
$$k_{D_1}(y,y')\leq \int_{\gamma[y,y']}\frac{|dw|}{\delta_{D_1}(w)}
\leq4\Big(\frac{r+2}{4r-1}\Big)\int_{\gamma[y,y']}\frac{|dw|}{\delta_D(w)}=4\Big(\frac{r+2}{4r-1}\Big)k_{D}(y,y'),$$

% $$k_{D_1}(z,z')\leq 8\Big(\frac{r+1}{8r-1}\Big)k_{D}(z,z')$$
\noindent and
$$k_{D_1}(y',z')\leq \pi.$$

Hence, the inequalities
\begin{eqnarray*}k_{D_1}(y,z)&\!\! \le \!\!& k_{D_1}(y,y') + k_{D_1}(y',z') + k_{D_1}(z',z)\\
    &\!\! \le \!\!& 4\Big(\frac{r+2}{4r-1}\Big)k_{D}(y,y')+\pi+4\Big(\frac{r+2}{4r-1}\Big)k_{D}(z,z')\\
&\!\! \le \!\!&4\Big(\frac{r+2}{4r-1}\Big)k_{D}(z,y)+4
\end{eqnarray*}
together with
\begin{eqnarray*} k_D(y,z)&\!\! = \!\!&\int_{\gamma[y,z]}\frac{|dw|}{\delta_D(w)} \leq k_D(y,y')+k_D(z',z)\\
&\!\! \le \!\!& \log\Big(1+\frac{|y-y'|}{\delta_D(y')}\Big)+\log\Big(1+\frac{|z-z'|}{\delta_D(z')}\Big)\\&\!\! \le \!\!&
2\log\Big(1+\frac{\frac{r}{2}\delta_D(x)}{\delta_D(x)+\frac{r}{2}\delta_D(x)}\Big)
\\&\!\! \le \!\!&2\log\Big(\frac{1+r}{1+\frac{r}{2}}\Big)
\end{eqnarray*}
give
\begin{eqnarray*}k_{D_1}(y,z)&\!\! \le \!\!&4\Big(\frac{r+2}{4r-1}\Big)k_{D}(z,y)+4\\&\!\! \le \!\!&
\Big(4\Big(\frac{r+2}{4r-1}\Big)+\frac{2}{\log(\frac{1+r}{1+\frac{r}{2}})}\Big)k_{D}(y,z).
\end{eqnarray*}

{\em Case II:} \;\;  $U=\emptyset$.

By \eqref{eq1'} we have
$$k_{D_1}(y,z)\leq 4\Big(\frac{r+2}{4r-1}\Big)k_{D}(y,z).$$

We finished the proof with $a(r)=\Big(4\Big(\frac{r+2}{4r-1}\Big)+\frac{2}{\log\frac{2+2r}{2+r}}\Big).$
\end{proof}

\begin{lemma}\label{genVu2-lem}
  For $\alpha,\theta\in(0,1)$, we define
  \[
    a(\alpha,\theta) = \frac{2+\theta+\alpha\theta}{\theta(1-\alpha^2)}
+\frac{(1+\alpha)\pi}{2(1-\alpha)\log ((2+2\theta)/(2+\theta+\alpha \theta))}\,.
  \]
If $x,y,z\in G$ with $x,y\in G\setminus B^n(z,\theta \delta_G(z))$, then
  \[
    k_{G'}(x,y)\le a(\alpha,\theta)\,k_G(x,y)\,,
  \]
  where $G' = G\setminus \overline{B}^n(z,\alpha \theta d(z,\partial G))$.
\end{lemma}
\begin{proof}
 Denote by $\delta(z)=d(z,\partial G)$.
 Fix $\beta \in (0,1)$ and $w \in G \setminus \overline{B}^n(z,\beta \delta(z))$. Choose
  $q\in S^{n-1}(z,\alpha \beta \delta(z))$ and $p\in\partial G$ such that
  $|w-q|=d(w,S^{n-1}(z,\alpha \beta \delta(z)))$ and $|p-z|=\delta(z)$.
Then we have
  $|w-q|\ge \beta(1-\alpha)\delta(z)$ and hence
  $$|p-q|\le (1+\alpha\beta)\delta(z)\le \frac{1+\alpha\beta}{\beta(1-\alpha)}|w-q|\,.
  $$
  It follows by the triangle inequality $|w-p|\le |p-q|+|w-q|$ that
  \begin{equation}\label{estimateofd}
    d(w,\partial G)\le |w-p| \le \frac{1+1/\beta}{1-\alpha}d(w,\partial G \cup S^{n-1}(z,\alpha \beta \delta(z)))\,.
  \end{equation}
  Let $J$ be a geodesic joining $x$ and $y$ in $G$ (i.e. $J=J_G[x,y]$)
  and $U = J \cap \overline{B}^n(z,(1+\alpha)\theta\delta(z)/2)$.

  If $U \neq \emptyset$ then we denote by $x'$ the first point in $U$, when we traverse along $J$ from $x$ to $y$.
  We similarly define $y'$ in $U\,,$ but traversing from $y$ to $x$ along $J$. By (\ref{estimateofd}) and Lemma
\ref{complementofB}
  \begin{eqnarray*}
    k_{G'}(x,y) &\!\! \le \!\!& k_{G'}(x,x')+k_{G'}(x',y')+k_{G'}(y',y) \\
    &\!\! \le \!\!& \frac{2+\theta+\alpha\theta}{\theta(1-\alpha^2)}k_G(x,x') + \frac{1+\alpha}{1-\alpha}\pi
+ \frac{2+\theta+\alpha\theta}{\theta(1-\alpha^2)}k_G(y',y) \\
    &\!\! \le \!\!& \frac{2+\theta+\alpha\theta}{\theta(1-\alpha^2)}k_G(x,y) + \frac{1+\alpha}{1-\alpha}\pi\,.
  \end{eqnarray*}
  Since $k_G \ge j_G$ we have
  \[
    k_{G}(x,y) \ge k_G(x,x')+k_G(y',y) \ge 2 \log \left( 1+\frac{\theta-(1+\alpha)\theta/2}{1+(1+\alpha)\theta/2} \right)
= 2 \log \frac{2+2\theta}{2+\theta+\alpha\theta}
  \]
  and therefore
  \[
    k_{G'}(x,y) \le a(\alpha,\theta) k_G(x,y)
  \]
holds for
  \[
    a(\alpha,\theta) = \frac{2+\theta+\alpha\theta}{\theta(1-\alpha^2)}+\frac{(1+\alpha)\pi}{2(1-\alpha)
\log ((2+2\theta)/(2+\theta+\alpha \theta))}\,.
  \]

  If $U = \emptyset$, then by (\ref{estimateofd})
  \[
    k_{G'}(x,y) \le \frac{2+\theta+\alpha\theta}{\theta(1-\alpha^2)}k_G(x,y) \le a(\alpha,\theta)k_G(x,y)\,.
  \]
  The assertion follows.
\end{proof}

Clearly Lemma \ref{genVu2-lem} implies Lemma \ref{Vu2-lem} as $\alpha \to 0$.

%%%%%%%%%%%%%%%%%%%%%%%%%%%%%%%%
%%%%%%%%%%%%%%%%%%%%%%%%%%%%%%%%%
%%%%%%%%%%%%%%%%%%%%%%%%%%%%%%%%%
\section{Examples of $\varphi$-uniform domains}\label{Sec3}

In order to give examples of $\varphi$-uniform domains,
consider domains $G$ satisfying the following geometric property \cite[Examples~2.50~(1)]{Vu2}:
there exists a constant $C\ge 1$ such that each pair of points
$x,y\in G$ can be joined by a rectifiable path $\gamma\in G$ with
$\ell(\gamma)\le C\,|x-y|$ and $\min\{\delta(x),\delta(y)\}\le
C\,d(\gamma,\partial G)$. Then $G$ is $\varphi$-uniform with
$\varphi(t)=C^2t$. In particular, every convex domain is
$\varphi$-uniform with $\varphi(t)=t$. However, in general, convex
domains need not be uniform.
%The above examples of $\varphi$-uniform domains are studied in \cite[Examples~2.50~(1)]{Vu2}.
More complicated nontrivial examples of $\varphi$-uniform domains can be
seen by considering that of uniform domains which are extensively studied by
several researchers. For instance, it is noted in \cite{Mac96} %\cite{Mac96,Va84}
that complementary components of quasim\"obius (and hence
bi-Lipschitz) spheres are uniform.

\subsection*{Complementary domains}
In this subsection we understand $\Rn\setminus \overline{D}$, for the terminology {\em complementary domain}
of a domain $D\subset\Rn$.
When we talk about complement of a domain is another domain, we mean in the sense of its
complementary domain.
Because simply connected uniform domains in plane are quasidisks \cite{MS79} (see also \cite{Geh99}),
it follows that the complement of such a uniform domain also is uniform.
A motivation to this observation of uniform domains leads to investigate
the complementary domains in the case of $\varphi$-uniform domains.
In fact we see from the following examples that complementary domains
of $\varphi_1$-uniform domains are not always $\varphi$-uniform for any $\varphi$.
The first example investigates the matter in the case of half-strips.
\begin{example}\label{half-strip}
Since the half-strip defined by
$S=\{(x,y)\in \Rt:\, x>0,~ -1<y<1\}$
is convex, by the above discussion we observe that it is $\varphi$-uniform
with $\varphi(t)=t$. On the other hand, by considering the points
$z_n=(n,-2)$ and $w_n=(n,2)$ we see that $G:=\Rt\setminus \overline{S}$
is not a $\varphi$-uniform domain. Indeed, we have
$j_G(z_n,w_n)=\log 5$ and for some $m\in \R\cap J_G[z_n,w_n]$
$$k_G(z_n,w_n)\ge k_G(m,w_n)\ge \log\left(1+\frac{|m-w_n|}{\delta(w_n)}\right)
\ge \log(1+n)\to \infty \quad\mbox{as $n\to\infty$}.
$$
This shows that $G$ is not $\varphi$-uniform for any $\varphi$.
\hfill{$\triangle$}
\end{example}

The above example gives a convex $\varphi$-uniform domain whose
complement is not $\psi$-uniform for any $\psi$.
We can construct a number of examples of $\varphi_1$-uniform domains,
whose complement is not a $\varphi$-uniform for any $\varphi$, by suitable
changes in the shape of the boundaries of the convex domains of above type.
For instance, we have the following example which provides
a $\varphi$-uniform domain (not convex) whose complement is not $\psi$-uniform for
any $\psi$.
\begin{example}\label{phiunif-exp2}
Define
$$D_m=\left\{(x,y)\in \Rt:\, |x|<\frac{1}{1+\log m},~0<y<\frac{me}{10}\right\}.
$$
It is clear by \cite[2.50]{Vu2} that the domain $D=\bigcup_{m=1}^\infty D_m$ is
$\varphi$-uniform with $\varphi(t)=2t$.
On the other hand, a similar reasoning explained in Example~\ref{half-strip} gives that
its complement $G'=\Rt\setminus \overline{D}$, is not $\psi$-uniform for
any $\psi$ (see Figure~\ref{ksv-fig1}).
\begin{figure}
\centering
\includegraphics{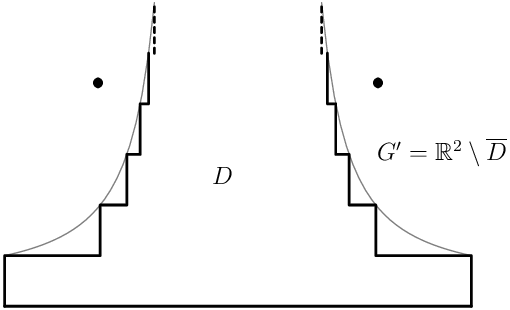}
\caption{An unbounded $\varphi$-uniform domain (not convex) $D\subset\Rt$ whose complement
$G'=\Rt\setminus\overline{D}$
is not $\psi$-uniform for any $\psi$. \label{ksv-fig1}}
\end{figure}
\hfill{$\triangle$}
\end{example}

We see that the $\varphi$-uniform domains considered in the above two examples are unbounded,
which generally asks the following problem:
\begin{problem}\label{phiunif-prob1}
Are there any bounded $\varphi_1$-uniform domains whose complementary domains are not
$\varphi$-uniform for any $\varphi$?
\end{problem}

The speciality in dimension 2, for Problem~\ref{phiunif-prob1}, is much more delicate.
Indeed, one can handle this matter by considering the domain by pulling thinner and
thinner rectangles from one edge of a rectangle (e.g. see Figure~\ref{newfig}).
\begin{figure}
\centering
\includegraphics[width=5cm]{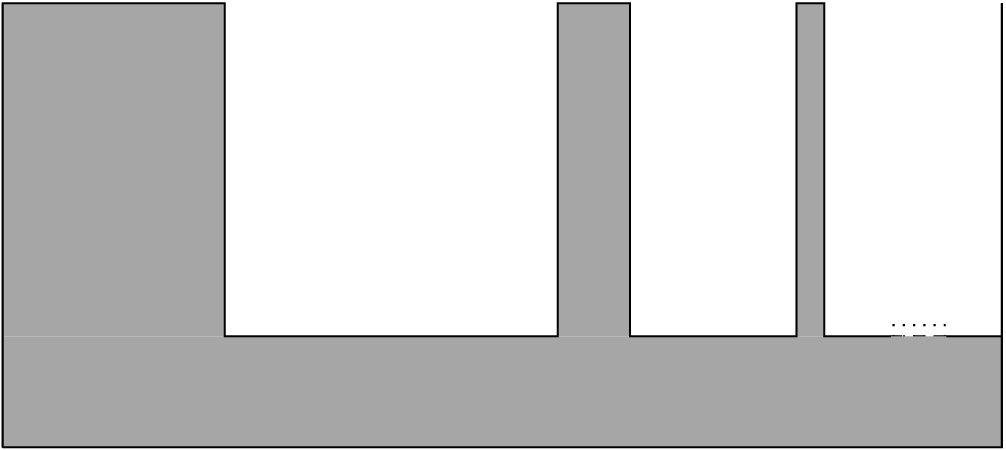}
\caption{A bounded $\varphi$-uniform domain whose complementary domain is not $\psi$-uniform for any $\psi$.}\label{newfig}
({\em The authors thank P. H\"ast\"o and one of the referees who have given the idea of this domain.})
\end{figure}
In this setting, one can
even similarly find Jordan domains which are $\varphi$-uniform but their complements
are not $\psi$-uniform for any $\psi$.
But we are not studying detail on it in this paper. However, it is sometimes
interesting to see examples in higher dimensional setting.

In three dimensional setting, we now provide solutions to Problem~\ref{phiunif-prob1} as follows:
\begin{example}\label{revo-ex1}
Let $T$ be the triangle with vertices $(1,-1)$, $(0,0)$ and $(1,1)$. Consider the
domain $D$ bounded by the surface of revolution generated by revolving $T$ about the
vertical axis (see Figure~\ref{kvs-fig3}).
%\begin{figure}
%\centering
%\includegraphics[width=4cm]{revolution.eps}
%\caption{A bounded $\varphi$-uniform domain in $\R^3$ whose complementary domain is not $\varphi$-uniform
%\label{kvs-fig3}}
%\end{figure}

%\begin{figure}
%\centering
%\includegraphics[width=6cm]{revolution3.eps}
%\caption{A bounded $\varphi$-uniform domain in $\R^3$ whose complementary domain is not $\varphi$-uniform
%\label{kvs-fig3}}
%\end{figure}

\begin{figure}
\begin{minipage}{0.45\linewidth}
\centering
\includegraphics[width=6cm]{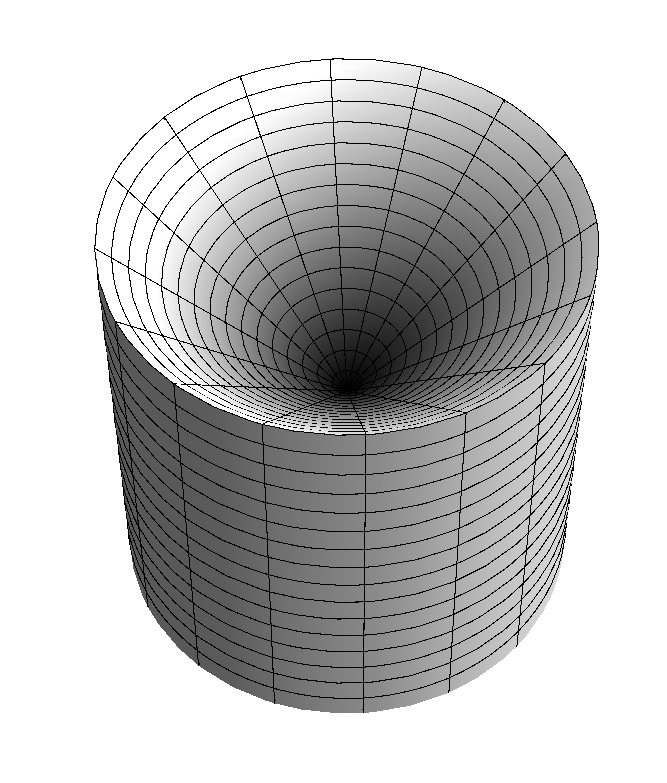}
\end{minipage}
\begin{minipage}{0.45\linewidth}
\centering
\includegraphics[width=6cm]{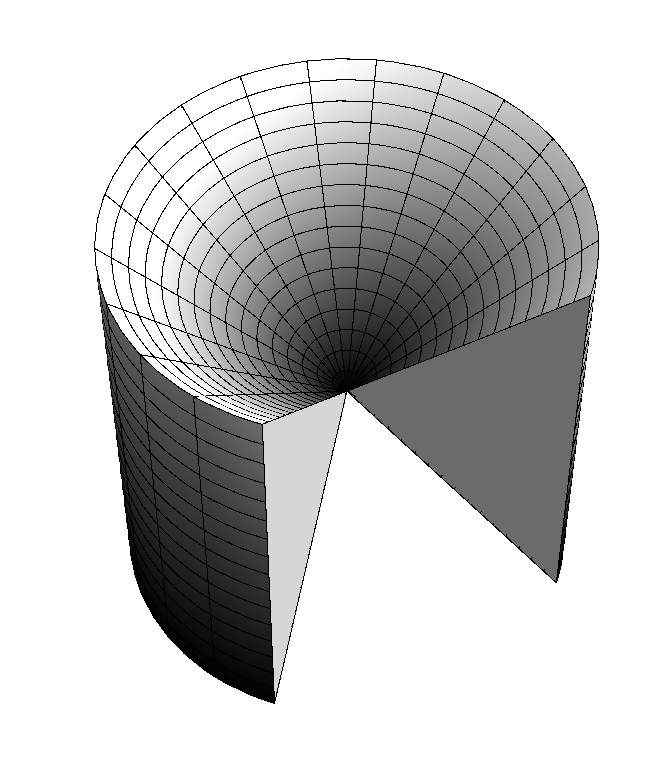}
\end{minipage}
\caption{A bounded $\varphi_1$-uniform domain in $\R^3$
whose complementary domain is not $\varphi$-uniform
(the right hand side figure shows the revolution).
\label{kvs-fig3}}
\end{figure}
Then we see that $D$ is $\varphi_1$-uniform for some $\varphi_1$ (in fact, $D$ is uniform\,!).
Indeed, let $x,y\in D$ be arbitrary. Without loss
of generality we assume that $|x|\ge |y|$. Consider the path $\gamma=[x,x']\cup C$ joining
$x$ and $y$, where $x'\in S^1(|y|)$ is chosen so that $|x'-x|=d(x,S^1(|y|))$; and $C$ is the smaller
circular arc of $S^1(|y|)$ joining $x'$ to $y$. For $z\in D$, we write $\delta(z):=d(z,\partial D)$.
Then for all $x,y\in D$ we have
\begin{eqnarray*}
k_D(x,y)\le \int_\gamma \frac{|dz|}{\delta(z)}
&\!\! = \!\!& \int_{[x,x']} \frac{|dz|}{\delta(z)}+ \int_C \frac{|dz|}{\delta(z)}\\
&\!\! \le \!\!& \frac{|x-y|}{\min\{\delta(x),\delta(y)\}}+ \int_C \frac{|dz|}{\delta(z)}\\
&\!\! \le \!\!& \left(1+\frac{\pi}{2}\right)\frac{|x-y|}{\min\{\delta(x),\delta(y)\}},
\end{eqnarray*}
where the last inequality follows by the fact that $\ell(C)\le \pi |x-y|/2$
(see the proof of Lemma~\ref{complementofB}).

On the other hand,
its complementary domain $G={\mathbb R}^3\setminus \overline{D}$ is not $\varphi$-uniform for any $\varphi$.
This can be easily seen by the choice $z_t=te_2\in G$, $0<t<1$. Indeed, we have $j_G(-z_t,z_t)=\log (1+2\sqrt{2})$;
and a similar argument as in Example~\ref{half-strip} leads
$$ k_G(-z_t,z_t)\ge \log\Big(1+\frac{\sqrt{2}}{t}\Big)\to \infty \quad\mbox{as $t\to 0$.}
$$
The assertion follows.
\hfill{$\triangle$}
\end{example}

Example~\ref{revo-ex1} provides a bounded $\varphi$-uniform domain in $\R^3$ which is
not simply connected. In the following, we construct a bounded simply
connected domain in $\R^3$ which is $\varphi$-uniform but its complement is not.

\begin{example}\label{doublecone}
Fix $h=1/3$. For the sake of convenience, we denote the coordinate axes in $\R^3$
by $x$-, $y$- and $z$-axes. Let $D$ be a domain obtained by rotating the triangle with vertices
$(0,0,h)$, $(1,0,0)$  and $(0,0,-h)$ around the $z$-axis. For each $k\ge 1$, we let
$$x_k=1-4^{-k} ~~\mbox{ and }~~ h_k=(1-x_k)/(10h)\,.
$$

We now modify the boundary of $D$ as follows: let us drill the cavity of
$D$ from two opposite directions of $z$-axis such that the drilling axis,
parallel to $z$-axis, goes through the point $(x_k,0,0)$. From the positive
direction we drill until the tip of the drill is at the height $h_k$ and from
the opposite direction we drill up to the height $-h_k$\,. The cross section
(see the left hand side of Figure~\ref{kvs-fig3double}) of
the upper conical surface will have its tip at $(x_k,y_k)$ described by
$$y-h_k=A(x-x_k),\quad A=\pm 1\,.
$$
This intersects the boundary of the cavity represented by $z=h(1-x)$ at
$$x=\frac{h-h_k+Ax_k}{A+h}\,.
$$
This gives
$$u_k=x|_{A=-1}=\frac{h-h_k-x_k}{-1+h}
~~\mbox{ and }~~
v_k=x|_{A=1}=\frac{h-h_k+x_k}{1+h}\,.
$$
Obviously, $u_k\le x_k \le v_k$. Since $v_k\le u_{k+1}$, we see that
two successive drilling do not interfere.

\begin{figure}
\begin{minipage}{0.45\linewidth}
\centering
\includegraphics[height=5cm]{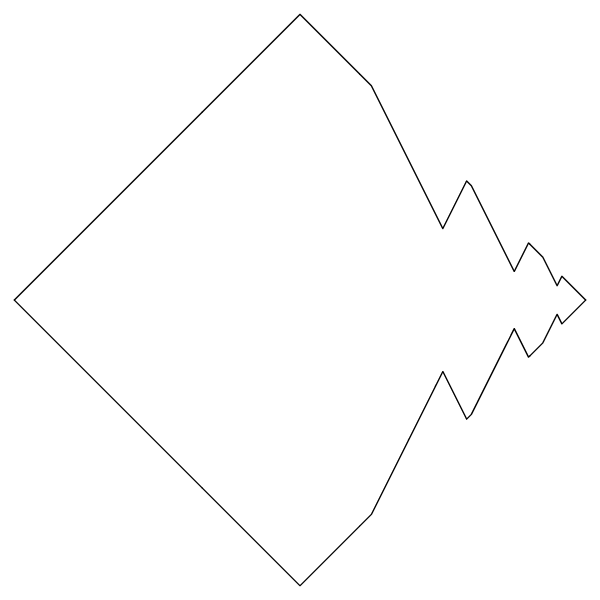}
\end{minipage}
\begin{minipage}{0.45\linewidth}
\centering
\includegraphics[height=5cm]{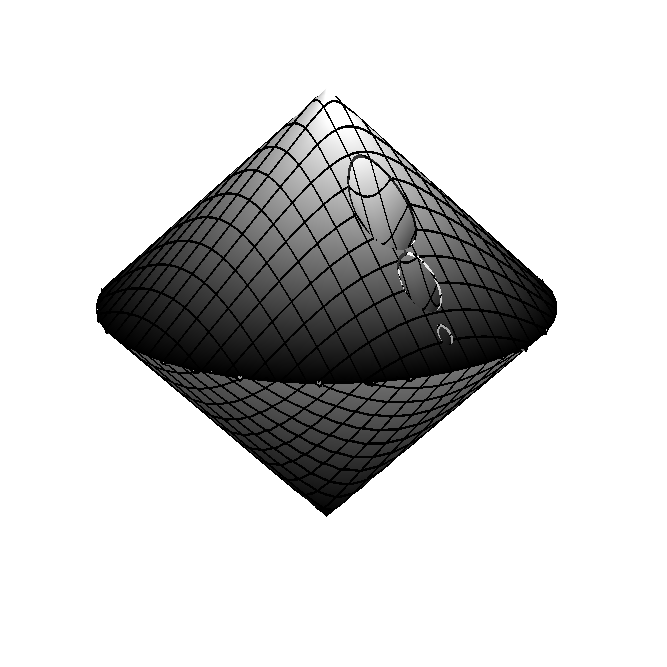}
\end{minipage}
\caption{ A double cone domain with two-sided drillings. The right hand side picture
provides a schematic view of the simply connected domain $G\subset
\mathbb {R}^3$ constructed in Example \ref{doublecone}. The left hand side picture is
a cross section of $G$. The domain $G$ is uniform but its complement is not $\varphi$-uniform
for any $\varphi$. \label{kvs-fig3double}}
\end{figure}

Induction on $k$ gives us a new domain $G\subset \R^3$ which is simply connected
and uniform, but its complement is not $\varphi$-uniform for any $\varphi$\,.
Indeed, the choice of points $z_k=(x_k,0,2h_k)$ and $w_k=(x_k,0,-2h_k)$ in
$U=\R^3\setminus\overline{G}$ gives that
$$d(z_k,\partial U)=\min\{x_k-u_k, v_k-x_k\}=h_kf(h)\,.
$$
It follows that
$$j_U(z_k,w_k)=\log\left(1+\frac{4h_k}{h_kf(h)}\right)< \infty\,.
$$
On the other hand,
$$k_U(z_k,w_k)\ge \log\left(1+\frac{\sqrt{1+h^2}}{d(z_k,\partial U)}\right)
=\log\left(1+\frac{\sqrt{1+h^2}}{h_kf(h)}\right)\to \infty \quad\mbox{ as $k\to \infty$}\,.
$$
This shows that $U$ is not $\varphi$-uniform for any $\varphi$.
\hfill{$\triangle$}
\end{example}

%\begin{openproblem} Are there any bounded planar $\varphi$-uniform domains
%whose complementary domains are not $\varphi$-uniform?
%\end{openproblem}
\subsection*{Quasiconvex domains}
A domain $G\subset \Rn$ is said to be {\em quasiconvex} if there exists a constant $c>0$ such that
every pair of points $x,y\in G$ can be joined by a rectifiable path $\gamma\subset G$ satisfying
$\ell(\gamma)\le c\,|x-y|$. We observe that the domains $G$ and $G'$ respectively in Examples
\ref{half-strip} and \ref{phiunif-exp2}
%Examples~\ref{phiunif-exp1} and \ref{phiunif-exp2}
are not quasiconvex and are not bounded too. This naturally leads to the
following problems.

\begin{problem}\label{phiunif-prob2}
Is it true that quasiconvex domains are $\varphi$-uniform and vice versa?
\end{problem}

We have a partial solution to Problem \ref{phiunif-prob2}, which is described in the following
example. This shows that there exist quasiconvex domains which are not $\varphi$-uniform for any $\varphi$.
%It is not clear to us whether $\varphi$-uniform domains are quasiconvex.
% \begin{example}
% Consider the unit square $D=\{(x,y)\in \Rt:\,|x|<1,|y|<1\}$. For $n\ge 1$ we define the set
% $$P_0^n=\left\{\Big(0,\pm\frac{1}{2^n}\Big),\Big(\pm\frac{1}{2^n},0\Big)\right\},
% $$
% and for $1\le m\le n-1$ we define
% $$P_m^n=\left\{\Big(\sum_{k=1}^m \frac{1}{2^k},\sum_{k=1}^m \frac{1}{2^k}\pm\frac{1}{2^n}\Big),
% \Big(\sum_{k=1}^m \frac{1}{2^k}\pm\frac{1}{2^n},\sum_{k=1}^m \frac{1}{2^k}\Big)\right\}.
% $$
% %Note that $P_m^n\to P_0^n$ when $m\to\infty$.
% Then consider the domain (see Figure~\ref{ksv-fig5})
% \begin{figure}
% \centering
% % \vspace*{-3cm}
% \includegraphics[width=4cm]{ksv-fig5.eps} %\vspace*{-5cm}
% \caption{A quasiconvex planar domain $G$ which is not $\varphi$ uniform for any $\varphi$\,. \label{ksv-fig5}}
% \end{figure}
% defined by
% $$G:=D\setminus \bigcup_{m=0}^k\bigcup_{n=m+1}^\infty P_m^n.
% $$
% When $k$ is large enough we observe that $G$ is not $\varphi$-uniform for any $\varphi$
% (we removed infinitely many points from the unit square $D$), but it is certainly
% a quasiconvex domain.
% \hfill{$\triangle$}
% \end{example}

\begin{example}
Start out with a construction of a finite point set on the boundary of the
unit square   $Q= [-1,1]\times [-1,1]$. For a fixed integer  $m >2$, we put on
$\partial Q$ so many evenly spaced
points with distance $2d$, their union is $E$, such that for every point in
$w \in \Rt \setminus Q$
we have
$$k_{\Rt \setminus E}(0,w) > m\,.
$$
Then obviously it is enough to choose $d$ such that
\begin{equation}\label{qc-examp}
k_{\Rt \setminus E}(0,w) \ge  k_{\Rt \setminus E}(0,p)\ge \log\Big(\frac{|p|}{{\rm dist}\,(p,E)}\Big)\ge  \log(1/d) > m\,,
\end{equation}
where $p$ is the point of intersection of $\partial Q$ and the geodesic segment from $0$ to $w$ in $\Rt \setminus E$.
We say that such a set $E$ is of type $m$.

Now we choose sequence of sets  $E_m$,
$m=1,2,3,\ldots$
such that each set is via a similarity transformation (i.e. a function $f$ of the form $|f(x)-f(y)|=c\,|x-y|$) equivalent to a
set of type $m$, and
that the sets behave as in Figure~\ref{ksv-fig5}
\begin{figure}
\centering
% \vspace*{-3cm}
\includegraphics[width=4cm]{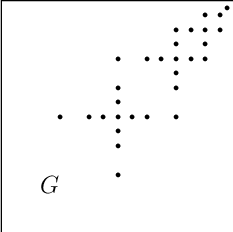} %\vspace*{-5cm}
\caption{A quasiconvex planar domain $G$ which is not $\varphi$ uniform for any $\varphi$\,. \label{ksv-fig5}}
\end{figure}
(i.e. converge to a corner of the square and are linked with each other at
the corner points and that diameter of $E_m$ is  $c \sqrt{2}  2^{-m}$, where $c$ is the constant of similarity transformation).
We denote $G:=Q \setminus   \cup_{m=1}^{\infty} \{ E_m   \}$.

Let $w_m$ be the center of the square, on whose boundary, the points of $E_m$ are located.
Then
$$\frac{|w_m - w_{m+1}|}{\min\{\delta_G(w_m),\delta_G(w_{m+1})\}}
=\frac{c\sqrt{2}(2^{-m}+2^{-(m+1)})}{(c\,2^{-m-1})/2}=6\sqrt{2}\,,
% =   (1/2)( c \sqrt{2}
% 2^{-m})   +   c \sqrt{2}  2^{-{m+1}}))/  (   c  2^{-m} ) =
% (1/2)(  \sqrt{2}   +  \sqrt{2}/2 ) =   3 \sqrt{2}/4
$$
while by a similar argument as in (\ref{qc-examp}) we get $k_G(w_m, w_{m+1})\ge k_{\Rt \setminus E_m}(w_m, w_{m+1})>m$.

Thus,  $G$ is not  $\varphi$-uniform
for any $\varphi$ but clearly it is quasiconvex.
\hfill{$\triangle$}
\end{example}

\begin{openproblem} Does there exist a simply connected quasiconvex planar domain which is not
$\varphi$-uniform for any $\varphi$?
% Is it true that $\varphi$-uniform domains are quasiconvex?
\end{openproblem}

\subsection*{Bilipschitz Property}
It is well-known that uniform domains are preserved under bilipschitz mappings (see
for instance \cite[p. 37]{Vu}). We now extend this property to the class of
$\varphi$-uniform domains.
\begin{proposition}\label{bilipschitz}
Let $f:\,\Rn\to \Rn$ be an $L$-bilipschitz mapping, that is
$$
|x-y|/L\le |f(x)-f(y)|\le L|x-y|
$$
for all $x,y\in \Rn$. If $G\psubset \Rn$ is $\varphi$-uniform, then
$f(G)$ is $\varphi_1$-uniform with $\varphi_1(t)=L^2\varphi(L^2t)$.
\end{proposition}
\begin{proof}
We denote $\delta(z):=d(z,\partial G)$ and $\delta'(w):=d(w,\partial f(G))$.
Since $f$ is $L$-bilipschitz, it follows that
$$\delta(z)/L\le \delta'(f(z))\le L\,\delta(z)
$$
for all $z\in G$. Also, we have the following well-known relation (see for instance \cite[p. 37]{Vu})
$$k_G(x,y)/L^2\le k_{f(G)}(f(x),f(y))\le L^2 k_G(x,y)
$$
for all $x,y\in G$. Hence, $\varphi$-uniformity of $G$ yields
\begin{eqnarray*}
k_{f(G)}(f(x),f(y)) &\!\! \le \!\!& L^2 k_G(x,y)\\
&\!\! \le \!\!& L^2 \varphi(|x-y|/\min\{\delta(x),\delta(y)\})\\
&\!\! \le \!\!& L^2 \varphi(L^2\,|f(x)-f(y)|/\min\{\delta'(f(x)),\delta'(f(y))\})\,.
\end{eqnarray*}
This concludes our claim.
\end{proof}

A mapping $h:\,\Rnbar\to \Rnbar$ defined by
$$h(x)=a+\frac{r^2(x-a)}{|x-a|^2},~h(\infty)=a,~h(a)=\infty
$$
is called an {\em inversion} in the sphere $S^{n-1}(a,r)$ for $x,a\in \Rn$ and $r>0$.
We recall the following well-known identity from \cite[(1.5)]{Vu}
\begin{equation}\label{inversion}
|h(x)-h(y)|=\frac{r^2|x-y|}{|x-a|\,|y-a|},\quad x,y\in\Rn\setminus \{a\}\,.
\end{equation}
%A counterpart of this identity for chordal distances is also obtained in %\cite[Theorem~7.29]{AVV}.

We next show that $\varphi$-uniform domains are
preserved under inversion in a sphere.
\begin{corollary} %\label{ksv-prop1}
Let $z_0\in \Rn$ and $R>0$ be arbitrary. Denote by $h$ an inversion in $S^{n-1}(z_0,R)$.
For $0<m<M$, if $G\subset B^n(z_0,M)\setminus\overline{B}^n(z_0,m)$ is a
$\varphi$-uniform domain, then $h(G)$ is $\varphi_1$-uniform with
$\varphi_1(t)=(M/m)^2\varphi(M^2t/m^2)$.
\end{corollary}
\begin{proof}
We denote $\delta(z):=d(z,\partial G)$ and $\delta'(w):=d(w,\partial h(G))$.
Without loss of generality we can assume that $z_0=0$. By the assumption on $G$ we see that
$m\le |z|\le M$ for all $z\in G$. Hence, by the identity (\ref{inversion}) we have
$$R^2|x-y|/M^2\le |h(x)-h(y)|\le R^2|x-y|/m^2
$$
which implies
$$R^2\min\{\delta(x),\delta(y)\}/M^2 \le \min\{\delta'(h(x),\delta'(y))\}
\le R^2\min\{\delta(x),\delta(y)\}/m^2.
$$
It follows that
$$(m/M)^2 k_G(x,y)\le k_{h(G)}(h(x),h(y)) \le (M/m)^2 k_G(x,y)
$$
for all $x,y\in G$. Since $G$ is $\varphi$-uniform, by a similar argument as
in the proof of Proposition~\ref{bilipschitz} we conclude our assertion.
%\begin{eqnarray*}
%k_{h(G)}(h(x),h(y))
%&\!\! \le \!\!& (M/m)^2 k_G(x,y)\\
%&\!\! \le \!\!& (M/m)^2 \varphi (|x-y|/\min\{\delta(x),\delta(y)\})\\
%&\!\! \le \!\!& (M/m)^2\varphi \Big(\frac{M^2}{m^2}\frac{|h(x)-h(y)|}{\min\{\delta(x),\delta(y)\}}\Big)\,.
%\end{eqnarray*}
%The claim follows.
\end{proof}
\section{Stability properties of $\varphi$-uniform domains}

Various classes of domains have been studied in analysis (e.g. see \cite{HPS}).
For some classes, the removal of a finite number of points from a domain may yield
a domain no longer in this class \cite{HPS}. Here we will investigate cases when
this does not happen, i.e. the removal of a finite number of points results a domain of the same class.

\begin{theorem}\label{phiunif-1pt}
For a fixed $\theta\in (0,1)$, consider the function $a(\theta)$ defined as in Lemma~$\ref{Vu2-lem}$.
If $G\psubset \Rn$ is a $\varphi_1$-uniform domain and $z_0\in G$, then
$G\setminus\{z_0\}$ is $\varphi$-uniform for some $\varphi$ depending on $\varphi_1$ only.
Moreover, we have
$$\varphi(t)=2\max\left\{\frac{\pi}{\log 3}\log (1+3t),a(\theta/4)\varphi_1(3t)\right\}\,.
$$
%continuous strictly
%increasing function $\varphi_1$ defined on $(0,1)$ with $\varphi_1(0)=0$.
\end{theorem}
\begin{proof}
In this proof we denote by $\delta_1$ the Euclidean distance to the boundary of $G$ and
$\delta_2$ the Euclidean distance to that of $G\setminus\{z_0\}$.
Fix $\theta\in (0,1)$ and let $x,y\in G\setminus\{z_0\}$ be arbitrary.
We prove the theorem by considering three cases.

{\em Case I:} $x,y\in B^n(z_0,\theta\delta_1(z_0)/2)\setminus\{z_0\}$.

We see that
\begin{eqnarray*}
k_{G\setminus\{z_0\}}(x,y)
&\!\! = \!\!& k_{\Rn\setminus\{z_0\}}(x,y)\\
&\!\! \le \!\!& \frac{\pi}{\log 3} j_{\Rn\setminus\{z_0\}}(x,y)\\
&\!\! \le \!\! & \frac{\pi}{\log 3} j_{G\setminus\{z_0\}}(x,y)\,,
\end{eqnarray*}
where the equality follows (see \cite[page 38]{mo}) from the fact that
$z_0$ is the closest point for the geodesics $J_G[x,y] = J_{G \setminus {z_0}}[x,y]$
which are logarithmic spirals (or circular arcs) in $B^n(z_0,\theta \delta_1(z_0)/2)$
%\cite[Theorem~1.5]{Lin} (more generally from
and the second inequality is due to Lind\'en \cite[Theorem~1.6]{Lin}.
It follows that
\begin{equation} %\label{phiunif-eq1}
k_{G\setminus\{z_0\}}(x,y) \le  \varphi_2(|x-y|/\min\{\delta_2(x),\delta_2(y)\})
\end{equation}
for $\varphi_2(t)=\frac{\pi}{\log 3}\log(1+t)$.

{\em Case II:} $x,y\in G\setminus B^n(z_0,\theta\delta_1(z_0)/4)$.

Since $G$ is $\varphi_1$-uniform, using Lemma~\ref{Vu2-lem} we obtain
\begin{eqnarray*}
k_{G\setminus\{z_0\}}(x,y)
&\!\! \le \!\!& a(\theta/4)k_G(x,y)\\
&\!\! = \!\!& a(\theta/4) \varphi_1(|x-y|/\min\{\delta_1(x),\delta_1(y)\})\\
&\!\! \le \!\! & a(\theta/4) \varphi_1(|x-y|/\min\{\delta_2(x),\delta_2(y)\}),
\end{eqnarray*}
where the last inequality holds because $\delta_1\ge \delta_2$. This gives that
\begin{equation} %\label{phiunif-eq2}
k_{G\setminus\{z_0\}}(x,y) \le  \varphi_3(|x-y|/\min\{\delta_2(x),\delta_2(y)\})
\end{equation}
with $\varphi_3(t)=a(\theta/4)\varphi_1(t)$.

{\em Case III:} $x\in B^n(z_0,\theta\delta_1(z_0)/4)\setminus\{z_0\}$ and
$y\in G\setminus B^n(z_0,\theta\delta_1(z_0)/2)$\,.

There exists a quasihyperbolic geodesic joining $x$ and $y$
that intersects the boundary of $B^n(z_0,\theta\delta_1(z_0)/4)$.
Let an intersecting point be $m$. Along this geodesic we have the
following equality
\begin{equation} \label{phiunif-eq3}
k_{G\setminus\{z_0\}}(x,y)=k_{G\setminus\{z_0\}}(x,m)+k_{G\setminus\{z_0\}}(m,y)\,.
\end{equation}
%for some point $m\in  J\cap S^{n-1}(z_0,\theta\delta_1(z_0)/2)$.
Now, {\em Case I} and {\em Case II}
respectively give
$$k_{G\setminus\{z_0\}}(x,m)\le \varphi_2(|x-m|/\min\{\delta_2(x),\delta_2(m)\})
$$
and
$$k_{G\setminus\{z_0\}}(m,y)\le \varphi_3(|m-y|/\min\{\delta_2(m),\delta_2(y)\})\,.
$$
We note that $\max\{|x-m|,|m-y|\}\le 3|x-y|$ and $\delta_2(m)\ge \delta_2(x)$.
Also, $\varphi_2$ and $\varphi_3$ being monotone, from (\ref{phiunif-eq3}) we obtain
\begin{eqnarray*}
k_{G\setminus\{z_0\}}(x,y)
&\!\! \le \!\!& \varphi_2(3|x-y|/\min\{\delta_2(x),\delta_2(y)\})
+\varphi_3(3|x-y|/\min\{\delta_2(x),\delta_2(y)\})\\
&\!\! \le \!\!& 2\max\{\varphi_2(3|x-y|/\min\{\delta_2(x),\delta_2(y)\}),
\varphi_3(3|x-y|/\min\{\delta_2(x),\delta_2(y)\})\}\\
&\!\! = \!\!& \varphi_4(|x-y|/\min\{\delta_2(x),\delta_2(y)\})\},
\end{eqnarray*}
where $\varphi_4(t)=2\max\{\varphi_2(3t),\varphi_3(3t)\}$.

We verified all the cases, and hence our conclusion holds with $\varphi=\varphi_4$.
\end{proof}
\begin{corollary}
Suppose that $\{z_1,z_2,\ldots,z_m\}$ is a finite non-empty set of points in
a domain $G\psubset \Rn$. If $G$ is $\varphi_0$-uniform,
then $G\setminus\{z_1,z_2,\ldots,z_m\}$ is $\varphi$-uniform for some $\varphi$
depending on $\varphi_0$, $m$ and the distance
$\min\{d(z_i,\partial G),|z_i-z_j|\}$ with $i\neq j$, $i,j=1,2,\ldots,m$.
\end{corollary}
\begin{proof}
As a consequence of Theorem~\ref{phiunif-1pt}, proof follows by induction on $m$.
Indeed, we obtain
$$\varphi(t)=2^m a(\theta/2)^{m-1}\max\{\pi(1+3t)/\log 3,a(\theta/2)\varphi_0(3t)\},
$$
where $\theta=\min\{d(z_i,\partial G),|z_i-z_j|\}$ with $i\neq j$ and $i,j=1,2,\ldots,m$.
\end{proof}
The following property of uniform domains, first noticed by
V\"ais\"al\"a (see \cite[Theorem~5.4]{Va88}) in a different
approach, is a straightforward consequence of
Theorem~\ref{phiunif-1pt}. For convenient reference we record the
following Bernoulli inequality:
\begin{equation}\label{Bernoulli}
\log(1+at)\le a\log(1+t); ~a\ge 1,~t\ge 0\,.
\end{equation}

\begin{corollary}
Suppose that $\{z_1,z_2,\ldots,z_m\}$ is a finite non-empty set of points
in a uniform domain $G\psubset \Rn$. Then
$G'=G\setminus\{z_1,z_2,\ldots,z_m\}$ also is uniform. More
precisely if $(\ref{C-uniform})$ holds for $G$ with some constant $C,$
then it also holds for $G'$ with a constant $C'$ depending on
$C$ and $m$.
\end{corollary}
\begin{proof}
It is enough to consider the domain $G\setminus\{z_1\}$ when (\ref{C-uniform}) holds for
$G$ with some constant $C$.
We refer to the proof of Theorem~\ref{phiunif-1pt}. Our aim is to find a constant $C'$ such that
$$k_{G\setminus\{z_1\}}(x,y)\le C'\,j_{G\setminus\{z_1\}}(x,y)\,.
$$
From {\it Case I}, we have $C'=\pi/\log 3$. Since (\ref{C-uniform}) holds for
$G$ with the constant $C$, from {\it Case II} we get $C'=C\,a(\theta/2)$.

By {\it Case I} and {\it Case II}, we see that
\begin{eqnarray*}
k_{G\setminus\{z_1\}}(x,y) &\!\! = \!\!& k_{G\setminus\{z_1\}}(x,m)+k_{G\setminus\{z_1\}}(m,y)\\
&\!\! \le \!\!& \max\{\pi/\log 3,Ca(\theta/2)\}\,[j_{G\setminus\{z_1\}}(x,m)+j_{G\setminus\{z_1\}}(m,y)]\\
&\!\! \le \!\!& C'\,j_{G\setminus\{z_1\}}(x,y)\,,
\end{eqnarray*}
where $C'=6\,\max\{\pi/\log 3,C\,a(\theta/2)\}$. Note that the last inequality follows by
similar reasoning as in {\it Case III} %of Theorem~\ref{phiunif-1pt}
and by the Bernoulli inequality (\ref{Bernoulli}).

Inductively, we notice that uniformity constant for $G'$ is
$$6^m a(\theta/2)^{m-1}\max\{\pi/\log 3,C\,a(\theta/2)\}=6^m a(\theta/2)^m C\,,
$$
where $a(\theta)$ is defined as in Lemma~\ref{Vu2-lem} for $\theta\in (0,1)$.
\end{proof}

%\comment{ %BEGIN COMMENT

\begin{theorem}\label{phiunif-E}
Let $\theta\in(0,1)$ be fixed. Assume that $G\psubset\Rn$ is $\varphi_1$-uniform and $z_0\in G$.
If $E\subset B^n(z_0,\theta d(z_0,\partial G)/5)$ is a non-empty closed set such that
$\Rn\setminus E$ is $\varphi_2$-uniform, then $G\setminus E$ is $\varphi$-uniform
for $\varphi$ depending on $\varphi_1$ and $\varphi_2$. %and $\theta$.
\end{theorem}
\begin{proof}
In this proof we denote by $\delta_1$, $\delta_2$ and $\delta_3$ the Euclidean distances to
the boundary of $G$, $G\setminus E$ and $\Rn \setminus E$ respectively.
Let $\theta\in (0,1)$ and $x,y\in G\setminus E$ be arbitrary.
We subdivide the proof into several cases. %similar to that of Theorem~\ref{phiunif-1pt}.

{\em Case A:} $x,y\in G\setminus B^n(z_0,\theta\delta_1(z_0)/4)$.

Denote $G'$ as in Lemma~\ref{genVu2-lem} but with $\alpha=1/5$.
Then $\varphi_1$-uniformity of $G$ gives
\begin{eqnarray*}
k_{G\setminus E}(x,y)
&\!\! \le \!\!& k_{G'}(x,y)\\
&\!\! \le \!\!& a(1/5,\theta/4)k_G(x,y)\\
&\!\! \le \!\!& a(1/5,\theta/4) \varphi_1(|x-y|/\min\{\delta_1(x),\delta_1(y)\})\\
&\!\! \le \!\!& a(1/5,\theta/4) \varphi_1(|x-y|/\min\{\delta_2(x),\delta_2(y)\})\,,
\end{eqnarray*}
where the first inequality holds by the monotonicity property, second inequality
follows by Lemma~\ref{genVu2-lem} and last follows trivially.

{\em Case B:} $x,y\in B^n(z_0,\theta\delta_1(z_0)/2)\setminus E$.

If $x,y\in B^n(z_0,\theta\delta_1(z_0)/4)\setminus E$, then
the quasihyperbolic geodesic $J:=J_{G\setminus E}[x,y]$
%joining $x$ and $y$ w.r.t the domain $G\setminus E$
may entirely lie in $B^n(z_0,\theta\delta_1(z_0)/3)$
or may intersect the sphere $S^{n-1}(z_0,\theta\delta_1(z_0)/3)$.
This means that the shape of $J$ will depend on the shape of $E$.
So, we divide the case into two parts.

{\em Case B1:} $J\cap S^{n-1}(z_0,\theta\delta_1(z_0)/3)=\emptyset$.

Note that for all $z \in J$, the closest boundary points to $z$ are in $E$, 
and thus, $J$ is also the quasihyperbolic geodesic $J_{\Rn \setminus \{ E \}}[x,y]$. 
Since $\Rn\setminus E$ is $\varphi_2$-uniform we have
\begin{equation}\label{phiunif-eq4}
k_{G\setminus E}(x,y)=k_{\Rn\setminus E}(x,y)
\le \varphi_2(|x-y|/\min\{\delta_2(x),\delta_2(y)\})\,.
%=\varphi_2(|x-y|/\min\{\delta_2(x),\delta_2(y)\}).
\end{equation}

{\em Case B2:} $J \cap S^{n-1}(z_0,\theta\delta_1(z_0)/3)\neq \emptyset$.

To get a conclusion like in (\ref{phiunif-eq4}) it is enough to show that
\begin{equation}\label{phiunif-eq5}
k_{G\setminus E}(x,y)\le C\,k_{\Rn\setminus E}(x,y)
\end{equation}
for some constant $C>0$.

{\em Case B2a:} $x,y\in B^n(z_0,\theta\delta_1(z_0)/4)\setminus E$ and
$k_{\Rn\setminus E}(x,y)>\log\frac{3}{2}$.

Let $x_1$ be the first intersection point of $J$ with $S^{n-1}(z_0,\delta_1(z_0)/3)$
when we traverse along $J$ from $x$ to $y$. Similarly, we define $x_2$ when we traverse
from $y$ to $x$ (see Figure~\ref{ksv-fig4}).
\begin{figure}
\centering
\includegraphics{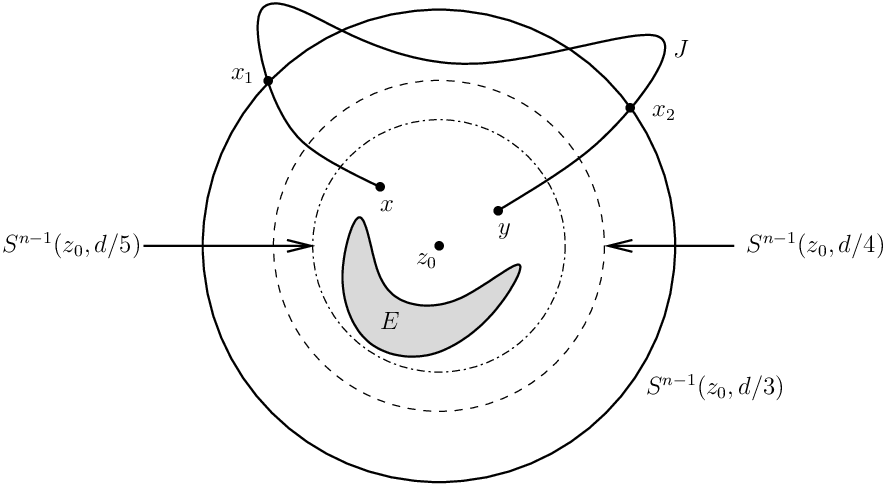}
\caption{The geodesic $J$ intersects $S^{n-1}(z_0,d/3)$ at $x_1$ and $x_2$, $d=\theta\delta_1(z_0)$.
\label{ksv-fig4}}
\end{figure}
In a similar fashion, let us denote $y_1$ and $y_2$ the first intersection points of
$J_{\Rn\setminus E}[x,y]$ with $S^{n-1}(z_0,\delta_1(z_0)/3)$ along both the directions
respectively. We observe that $\delta_2(z)=\delta_3(z)$ for all $z\in J[x,x_1]$, where
$J[x,x_1]$ denotes part of $J$ from $x$ to $x_1$. Hence,  along the geodesic $J$
we have
\begin{equation}\label{phiunif-eq6}
k_{G\setminus E}(x,y)=k_{\Rn\setminus E}(x,x_1)+k_{G\setminus E}(x_1,x_2)+k_{\Rn\setminus E}(x_2,y)\,.
\end{equation}
Now, by the triangle inequality we see that
\begin{equation}\label{phiunif-eq7}
k_{\Rn\setminus E}(x,x_1)\le k_{\Rn\setminus E}(x,y_1)+k_{\Rn\setminus E}(y_1,x_1)\,.
\end{equation}
By comparing the quasihyperbolic distance along the circular path joining $y_1$ and $x_1$,
we obtain
\begin{equation}\label{phiunif-eq8}
k_{\Rn\setminus E}(y_1,x_1)\le 4\pi\,.
\end{equation}
On the other hand, we see that
\begin{equation}\label{phiunif-eq9}
k_{\Rn\setminus E}(x,y_1)\ge j_{\Rn\setminus E}(x,y_1)\ge \log \frac{8}{7}\,,
\end{equation}
because $|x-y_1|\ge \theta\delta_1(z_0)/12$ and $\delta_3(y_1)\le 7\theta\delta_1(z_0)/12$.
Combining (\ref{phiunif-eq8}) and (\ref{phiunif-eq9}), from (\ref{phiunif-eq7}) we obtain
$$k_{\Rn\setminus E}(x,x_1)\le \Big(1+\frac{4\pi}{\log \frac{8}{7}}\Big)k_{\Rn\setminus E}(x,y_1)\,.
$$
Similarly we get
$$k_{\Rn\setminus E}(x_2,y)\le \left(1+\frac{4\pi}{\log \frac{8}{7}}\right)k_{\Rn\setminus E}(y_2,y)\,.
$$
A similar argument as in (\ref{phiunif-eq8}) and the last two inequalities together
with (\ref{phiunif-eq6}) give
$$k_{G\setminus E}(x,y)\le 4\pi+ \left(1+\frac{4\pi}{\log \frac{8}{7}}\right) k_{\Rn\setminus E}(x,y)\,.
$$
So, by the assumption in this case, the inequality (\ref{phiunif-eq5}) follows from the last inequality with
the constant $C=1+(4\pi/\log \frac{8}{7})+(4\pi/\log \frac{3}{2})$.

{\em Case B2b:} $x,y\in B^n(z_0,\theta\delta_1(z_0)/4)\setminus E$ and
$k_{\Rn\setminus E}(x,y)\le \log\frac{3}{2}$.

The well-known inequality $j_{\Rn\setminus E}(x,y)\le k_{\Rn\setminus E}(x,y)$ reduces to
\begin{equation}\label{phiunif-eq10}
R:=\frac{1}{2}\min\{\delta_3(x),\delta_3(y)\}\ge |x-y|\,.
\end{equation}
Without loss of generality we assume that $\min\{\delta_3(x),\delta_3(y)\}=\delta_3(x)$.
Then there exists a point $x_0\in S^{n-1}(x,2R)\cap \partial E$ such that $\delta_3(x)=|x-x_0|=2R$.
For the proof of (\ref{phiunif-eq5}), we proceed as follows
\begin{eqnarray*}
k_{G\setminus E}(x,y)
&\!\! \le \!\!& k_{B^n(x,2R)}(x,y)\\
&\!\! \le \!\!& 2 j_{B^n(x,2R)}(x,y)\\
&\!\! = \!\!& 2 \log\left(1+\frac{|x-y|}{\delta_3(x)-|x-y|}\right)\\
&\!\! \le \!\!& 2 \log\left(1+\frac{2|x-y|}{\delta_3(x)}\right)\\
&\!\! \le \!\!& 4 \log\left(1+\frac{|x-y|}{\delta_3(x)}\right)\\
&\!\! = \!\!& 4 j_{\Rn\setminus \{x_0\}}(x,y)\\
&\!\! \le \!\!& 4 j_{\Rn\setminus E}(x,y)\,,
\end{eqnarray*}
where the second, third and fourth inequalities follow from \cite[Lemma~7.56]{AVV}, (\ref{phiunif-eq10})
and (\ref{Bernoulli}) respectively. Hence we proved {\em Case B} when
$x,y\in B^n(z_0,\theta\delta_1(z_0)/4)\setminus E$.

If $x\in B^n(z_0,\theta\delta_1(z_0)/4)\setminus E$ and
$y\in B^n(z_0,\theta\delta_1(z_0)/3)\setminus \overline{B}^n(z_0,\theta\delta_1(z_0)/4)$,
by considering a sphere $S^{n-1}(z_0,\theta\delta_1(z_0)r)$ with $r\in(1/4,1/3)$ we
proceed like before.

If $x\in B^n(z_0,\theta\delta_1(z_0)/4)\setminus E$ and
$y\in B^n(z_0,\theta\delta_1(z_0)/2)\setminus \overline{B}^n(z_0,\theta\delta_1(z_0)/3)$,
then the geodesic $J_{G\setminus E}[x,y]$ will intersect $S^{n-1}(z_0,\theta\delta_1(z_0)/3)$.
Let $m$ be the first intersection point when we traverse along the geodesic from $x$ to $y$.
Then along the geodesic we have
\begin{eqnarray*}
k_{G\setminus E}(x,y)
&\!\! = \!\!& k_{G\setminus E}(x,m)+k_{G\setminus E}(m,y)\\
&\!\! \le \!\!& k_{\Rn\setminus E}(x,m)+
 a(1/4,\theta/3)\varphi_1(|m-y|/\min\{\delta_2(m),\delta_2(y)\})\\
&\!\! \le \!\!& \varphi_2(|x-m|/\min\{\delta_3(x),\delta_3(m)\})+
 a(1/4,\theta/3)\varphi_1(|m-y|/\min\{\delta_2(m),\delta_2(y)\})\\
&\!\! \le \!\!& \varphi_2(4|x-y|/\min\{\delta_2(x),\delta_2(y)\})+
 a(1/4,\theta/3)\varphi_1(10|x-y|/\min\{\delta_2(x),\delta_2(y)\})\,,
\end{eqnarray*}
where the first and second inequalities follow by {\em Case A} and the assumption on
$E$ respectively. Thus, we conclude that if $x,y\in B^n(z_0,\theta\delta_1(z_0)/2)\setminus E$,
then
%\begin{equation}\label{phiunif-eq11}
$$k_{G\setminus E}(x,y)\le 2a(1/4,\theta/3)\varphi_3(|x-y|/\min\{\delta_2(x),\delta_2(y)\})\,,
$$
%\end{equation}
with $\varphi_3(t)=\max\{\varphi_2(10t),\varphi_1(10t)\}$.

{\em Case C:} $x\in B^n(z_0,\theta\delta_1(z_0)/4)\setminus E$ and
$y\in G\setminus B^n(z_0,\theta\delta_1(z_0)/2)$.

Let $p\in J_{G\setminus E}[x,y]\cap S^{n-1}(z_0,\theta\delta_1(z_0)/4)$. Then we see that
\begin{eqnarray*}
k_{G\setminus E}(x,y)
&\!\! = \!\!& k_{G\setminus E}(x,p)+k_{G\setminus E}(p,y)\\
&\!\! \le \!\!& 2a(1/4,\theta/3)\varphi_3(|x-p|/\min\{\delta_2(x),\delta_2(p)\})\\
&&\qquad +a(1/5,\theta/4) \varphi_1(|p-y|/\min\{\delta_2(p),\delta_2(y)\})\\
&\!\! \le \!\!& 2a(1/4,\theta/3)\varphi_3(|x-p|/\min\{\delta_2(x),\delta_2(y)\})\\
&& \qquad +a(1/5,\theta/4) \varphi_1(|p-y|/\min\{\delta_2(x),\delta_2(y)\})\,,
\end{eqnarray*}
where the first inequality holds by {\em Case B} and {\em Case A}, and last
holds by a similar argument as above (or as in the proof of {\em Case III} in Theorem~\ref{phiunif-1pt}).
It is easy to see that
$$\max\{|x-p|,|p-y|\}\le 3|x-y|\,.
$$
In the same way, as in Theorem~\ref{phiunif-1pt}, the monotonicity property of $\varphi_3$ and
$\varphi_1$ gives
\begin{eqnarray*}
k_{G\setminus E}(x,y)
%&\!\! \le \!\!& 2\,a(1/4,\theta/3)\varphi_3(3t)+ a(1/5,\theta/4)\varphi_1(3t)\\
&\!\! \le \!\!& 4\,a(1/4,\theta/3)\max\{\varphi_2(30|x-y|/\min\{\delta_2(x),\delta_2(y)\}),\\
&&\hspace{4cm} \varphi_1(30|x-y|/\min\{\delta_2(x),\delta_2(y)\})\}.
\end{eqnarray*}

By combining all the above cases, a simple computation concludes that the domain $G\setminus E$ is $\varphi$-uniform
for $\varphi(t)=4\,a(1/4,\theta/3)\max\{\varphi_1(30t),\varphi_2(30t)\}$, where
$a(1/4,\theta/3)$ is obtained from Lemma~\ref{genVu2-lem}.
%$$ k_{G\setminus E}(x,y) \le 4\,a(1/4,\theta/3)\max\{\varphi_2(30t),\varphi_1(30t)\}
%$$
%for all $x,y\in G\setminus E$.
\end{proof}
\begin{corollary}\label{coro43}
Fix $\theta\in(0,1)$. Assume that $G\psubset\Rn$ is $\varphi_0$-uniform and $(z_i)_{i=1}^m$
are non-empty finite set of points in $G$ such that $\delta(z_1)=\min \{\delta(z_i)\}_{i=1}^m$.
Denote
$$ d:=\min_{i\neq j}\{|z_i-z_j|/2\} ~~\mbox{and}~~ \delta:=\min\{\delta(z_1),d\}\,.
$$
If all $E_i$, $i=1,2,\ldots,m$, are non-empty closed sets in
$B^n(z_i,\theta \delta/5)$ such that $\Rn\setminus \bigcup_{i=1}^m E_i$ is $\varphi_1$-uniform
for some $\varphi_1$, then the domain $G\setminus\bigcup_{i=1}^m E_i$ is $\varphi$-uniform for some $\varphi$.
\end{corollary}
\begin{proof}
As a consequence of Theorem~\ref{phiunif-E} the proof follows by
induction.
\end{proof}

%\bigskip
%($\star$ {\bf Remark: What follows are the results I added. Note that I use $\delta_D(\cdot)$ but not $\delta_1(\cdot)$})

What we consider above are removing finite number of points or sets from a domain in a class yields a domain in
the same class. In the following,
we would like to consider the case of removing infinite number of points or sets from a domain, but in the geometric sequence.
We first introduce a lemma (see \cite[Theorem 2.23]{Va98}) which is used latter in this section.

\begin{lemma}\label{lem4}
Suppose that $\gamma$ is a $c$-uniform arc in $D$ with end points $a,b$.
Then $$k_D(a,b)\leq 7c^3 \log\Big(1+\frac{|a-b|}{\delta_D(a)\wedge \delta_D(b)}\Big).
$$
\end{lemma}

%For convenience, in what follows, we always denote $\delta_D(\cdot):=d(\cdot, \partial D)$.

% Let $\{x_k\}_{k=1}^{\infty}$ be a sequence of points in $B^n(x_0,r)$ satisfying:  $x_k\in[x_0,x_{k-1})$
% and $|x_0-x_k|=\frac{1}{2^{k+2}}r.$ Denote
% \begin{equation}\label{setE}
%   E=\{x_0\}\cup\{x_k\}_{k=1}^{\infty}.
% \end{equation}
As a consequence of \cite[Theorem~5.4]{Va88}, one can prove that
\begin{theorem}\label{uniform1}
Let $\{x_k\}_{k=1}^{\infty}$ be a sequence of points in $B^n(x_0,r)$ satisfying:  $x_k\in[x_0,x_{k-1})$
and $|x_0-x_k|=\frac{1}{2^{k+2}}r.$ Denote
% \begin{equation}\label{setE}
$E=\{x_0\}\cup\{x_k\}_{k=1}^{\infty}\,.
$
% \end{equation}
Then there exists some constant $c$ such that $B^n(x_0,r)\setminus E$ is $c$-uniform.
% where $E$ is as in \eqref{setE}.
\end{theorem}

We now provide a similar result in the case of $\varphi$-uniform domains.

\begin{theorem}\label{psiunif2}
Suppose that $D\psubset \mathbb{R}^n$ is a $\varphi$-uniform domain and $x_0\in D$.
let $x_0\in D$, and $\{x_k\}_{k=1}^{\infty}$  be a sequence of points in
$B^n(x_0,\delta_D(x_0))$ satisfying:  $x_k\in[x_0,x_{k-1})$ and $|x_0-x_k|=\frac{1}{2^{k+2}}\delta_D(x_0).$
Denote $E=\{x_0\}\cup\{x_k\}_{k=1}^{\infty}$. Then $D\setminus E$
is $\psi$-uniform with $\psi$ depending on $\varphi$.
\end{theorem}

\begin{proof}
We note that $E\subseteq\overline{B}^n(x_0, \frac{1}{8}\delta_D(x_0))$. Let $x,y\in D\setminus E$
be arbitrary. We subdivide the proof into several cases.

{\em Case I:} $x,y\in B^n(x_0,\frac{1}{2}\delta_D(x_0))\setminus E.$

By Theorem \ref{uniform1} we know that $B^n(x_0,\frac{1}{2}\delta_D(x_0))\setminus E$ is $c$-uniform
with some constant $c$. Then we can join $x,y$ by a uniform arc $\gamma$ in $B^n(x_0,\frac{1}{2}\delta_D(x_0))\setminus E$,
hence $\gamma$ is uniform in $D\setminus E$ also. Lemma \ref{lem4} shows that $D\setminus E$
is $\varphi_2$-uniform with $\varphi_1(t)=c\log(1+t)$.

{\em Case II:} $x,y\in D\setminus B^n(x_0,\frac{1}{4}\delta_D(x_0)) .$

Since $D$ is $\varphi$-uniform, using Lemma~\ref{Lem1} we get
\begin{eqnarray*}k_{D\setminus E}(x,y)&\!\! \le \!\!& a\Big(\frac{1}{4}\Big)\,k_D(x,y)\\
&\!\! \le \!\!& a\Big(\frac{1}{4}\Big)\,\varphi\Big(\frac{|x-y|}{\min\{\delta_D(x), \delta_D(y\}}\Big)\\
&\!\! \le \!\!& a\Big(\frac{1}{4}\Big)
\varphi\Big(\frac{|x-y|}{\min\{\delta_{D\setminus E}(x), \delta_{D\setminus E}(y)\}}\Big).
\end{eqnarray*}
This gives that
$$k_{D\setminus E}(x,y)\leq \varphi_2\Big(\frac{|x-y|}{\min\{\delta_{D\setminus E}(x), \delta_{D\setminus E}(y)\}}\Big)
$$
with $\varphi_2(t)=a(\frac{1}{4})\varphi(t).$

{\em Case III:} $x\in B^n(x_0,\frac{1}{4}\delta_D(x_0))\setminus E,$ $y\in D\setminus B^n(x_0,\frac{1}{2}\delta_D(x_0)) .$

Let $w\in S^{n-1}(x_0, \frac{1}{2}\delta_D(x_0))$. Then we have $$\delta_{D\setminus E}(w)\geq
\frac{1}{4}\delta_D(x_0)\geq \delta_{D\setminus E}(x),$$
and
$$\max\{|x-w|,|y-w|\}\leq 5|x-y|.$$
Hence
\begin{eqnarray*}k_{D\setminus E}(x,y)&\!\! \le \!\!&k_{D\setminus E}(x,w)+k_{D\setminus E}(w,y)\\
&\!\! \le \!\!&\varphi_1\Big(\frac{|x-w|}{\min\{\delta_{D\setminus E}(x),
\delta_{D\setminus E}(w)\}}\Big)+\varphi_2\Big(\frac{|y-w|}{\min\{\delta_{D\setminus E}(y),
\delta_{D\setminus E}(w)\}}\Big)\\
&\!\! \le \!\!& 2\max\{ \varphi_1\Big(\frac{5|x-y|}{\min\{\delta_{D\setminus E}(x), \delta_{D\setminus E}(y)\}}\Big),
\varphi_2\Big(\frac{5|x-y|}{\min\{\delta_{D\setminus E}(x), \delta_{D\setminus E}(y)\}}\Big)\}\\
&\leq& \varphi_3\Big(\frac{|x-y|}{\min\{\delta_{D\setminus E}(x), \delta_{D\setminus E}(y)\}}\Big),
\end{eqnarray*}
where $\varphi_3(t)=2\max\{\varphi_1(5t),\varphi_2(5t)\}.$
Hence we complete the proof with $\psi=\varphi_3.$
\end{proof}

Although the following is a consequence of \cite[Theorem~5.4]{Va88}, it follows directly
from Theorems \ref{uniform1} and \ref{psiunif2}.
\begin{corollary}\label{cor12}Suppose that $D\subset \Rn$ is a $c$-uniform domain and $x_0\in D$.
Then $D\setminus E$ is $c_1$-uniform with $c_1$ depending on $c$, where $E$ is defined as in Theorem~\ref{psiunif2}.
\end{corollary}
% \begin{proof}Corollary \ref{cor12} follows directly from Theorem \ref{uniform1} and Theorem \ref{psiunif2}.
% \end{proof}

\begin{theorem}\label{paiunif3}
Let $\{x_i\}_{i=1}^{\infty}$ be a sequence of points in $B^n(x_0,r)$ satisfying
$x_i\in[x_0,x_{i-1})$ and $|x_0-x_i|=\frac{1}{2^{i+2}}r.$
Assume that $B_i$, $i=1,2,\ldots$, are disjoint balls with centers $x_i$ and radii $r_i$. For $c\in(0,1)$,
let $E_i\subset B^n(x_i,c r_i)$ be the closed sets whose complements with respect to $\Rn$ are $\varphi$-uniform
and satisfy $x_i\in E_i$. Denote $F=\cup E_i\cup\{x_0\}$.
Then $B^n(x_0,r)\setminus F$ is $\psi$-uniform with $\psi$
depending on $\varphi$.
\end{theorem}

\begin{proof} Without loss of generality, we may assume that $x_0=0, r=1$ and $c=\frac{1}{4}$.

Let $G=B^n\setminus F$, and $x,y\in G$ be arbitrary. We prove the theorem by considering three cases.

{\em Case A:}  $x,y\in G\setminus \overline{B}^n(\frac{3}{16})$.

By the choice of $F$, we have $F\subset B^n(\frac{3}{16})$. Then $G\setminus \overline{B}^n(\frac{3}{16})$
is a $c$-uniform domain with some constant $c$. Hence we can join $x$ and $y$ by an arc $\gamma$
in $G\setminus \overline{B}^n(\frac{3}{16})$ such that $\gamma$ is a uniform arc in $G$.
By Lemma \ref{lem4}, we see that for every $x,y\in G\setminus \overline{B}^n(\frac{3}{16})$,
$$k_G(x,y)\leq \varphi_1\Big(\frac{|x-y|}{\min\{\delta_G(x), \delta_G(y)\}}\Big)
$$
with $\varphi_1(t)=c\log(1+t).$

{\em Case B:}  $x,y\in B^n(\frac{1}{4})\setminus F$.

If $|x-y|\leq \frac{1}{2}\min\{\delta_G(x), \delta_G(y)\},$ then $$k_G(x,y)\leq \int_{[x,y]}
\frac{|dw|}{\delta_G(w)}\leq \frac{2|x-y|}{\min\{\delta_G(x), \delta_G(y)\}}.$$
If $|x-y|\geq \frac{1}{2}\min\{\delta_G(x), \delta_G(y)\},$ then we assume that $|x|\leq |y|$,
and we divide the discussion of the proof into three subcases.

{\em Case B1:}  $x,y\notin \cup^{\infty}_{i=1} B^n(x_i, \frac{1}{2}r_i).$

In this case, there exist some non-negative integers $s$ and $t$ with $s\geq t$ such that
$x\in B^n(\frac{3}{2^{s+2}})\setminus B^n(\frac{3}{2^{s+3}})$ and $y\in B^n(\frac{3}{2^{t+2}})
\setminus B^n(\frac{3}{2^{t+3}}).$
If $|x-y|\leq \frac{3}{2^{t+8}}$, then we have $s=t$ or $s=t+1.$ By corollary \ref{coro43}
we get $G_1=[B^n(\frac{3}{2^{t+1}})\setminus \overline{B}^n(\frac{3}{2^{t+5}})]\setminus F$ is $\varphi_0$-uniform,
where $\varphi_0$ depends on $\varphi$.
Hence
$$k_G(x,y)\leq k_{G_1}(x,y)\leq \varphi_0\Big(\frac{|x-y|}{\min\{\delta_{G_1}(x),
\delta_{G_1}(y)\}}\Big)\leq \varphi_0\Big(\frac{|x-y|}{\min\{\delta_{G}(x), \delta_{G}(y)\}}\Big),
$$
the last equality holds because $\delta_{G_1}(w)\geq \delta_G(w)$
for every $w\in B^n(\frac{3}{2^{t+2}})\setminus B^n(\frac{3}{2^{t+4}}).$

In the following, we consider the case $|x-y|\geq \frac{3}{2^{t+8}}.$ Let $T_x$, $T_y$
be $2$-dimensional subspaces determined by $x$ and $[0,x_1]$, $y$ and $[0,x_1]$, respectively.
Let $l$ denote the line determined by $0$ and $x_1$, then $0$ divides $l$ into two rays:
$l_1$ and $l_2$ with $x_1\in l_2$. Denote the points of intersection of $l_1$ with
$T_x\cap S^{n-1}(|x|)$ and with $T_y\cap S^{n-1}(|y|)$ by $p_x$ and $p_y$, respectively.
Then $x$ and $p_x$ determine a shorter arc (or semicircle) in circle $T_x\cap S^n(|x|)$
which is denoted by $\alpha$, similarly, $y$ and $p_y$ determine a shorter arc (or semicircle)
in circle $T_y\cap S^{n-1}(|y|)$ denoted by $\beta$. Let $\gamma=\alpha\cup[p_x,p_y]\cup\beta$
(see Figure \ref{fig2}),  and let $m$, $n$ be positive integers  such that
$\delta_G(x)=\delta_{\mathbb{R}^n\setminus E_m}(x)$, $\delta_G(y)=\delta_{\mathbb{R}^n\setminus E_n}(y)$
and denote $G_2=\mathbb{R}^n\setminus (E_m \cup E_n)$.

\begin{figure}
\centering
\includegraphics{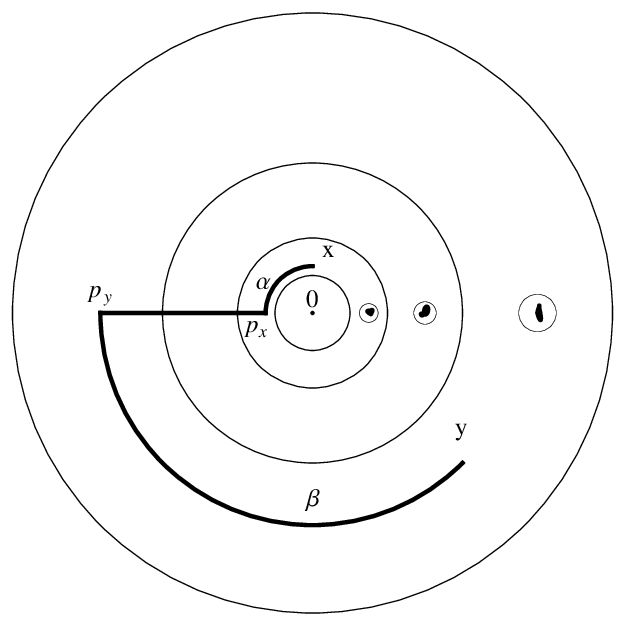}
\caption{Picture for \textit{Case B1}: The radii of the circles are $3/2^{s + 3}$, $3/2^{s + 2}$,
$3/2^{t + 3}$ and $3/2^{t + 2}$ respectively, and the small balls contained in the rings are
balls centered at $x_i$ and with radii $r_i/4$. We note that the removed sets $E_i' s$ are contained in such balls.
\label{fig2}}
\end{figure}

We now prove
\be\label{claim1}\delta_G(w)\geq \frac{1}{2}\min\{\delta_{G_2}(x), \delta_{G_2}(y)\}\quad \mbox{for every $w\in\gamma$}.
\ee
Let $p$ be a positive integer such that $\delta_G(w)=\delta_{\mathbb{R}^n\setminus E_p}(w)$.
Then for all $w\in \alpha$  we have
$$\delta_G(w)\geq\frac{1}{2}|w-x_k|\geq\frac{1}{2}|x_k-x|\geq \frac{1}{2}\delta_{G}(x)=\frac{1}{2}\delta_{G_2}(x).
$$
Similarly, for all $w\in \beta$, $\delta_G(w)\geq \frac{1}{2}\delta_{G_2}(y)$ holds.

If $w\in[p_x,p_y]$, then we get $\delta_{G}(w)\geq \delta_{G}(p_x)\geq\frac{1}{2}\delta_{G_2}(x).$
The proof of (\ref{claim1}) follows.
% \end{proof}

By (\ref{claim1}) and Corollary~\ref{coro43}, we get
\begin{eqnarray*}k_G(x,y)&\!\! \le \!\!& \int_{\gamma}\frac{|dw|}{\delta_G(w)}
\le 2 \int_{\gamma}\frac{|dw|}{\min\{\delta_{G_2}(x), \delta_{G_2}(y)\}}\\
&\!\! \le \!\!&2^6(2\pi+1)\frac{|x-y|}{\min\{\delta_{G_2}(x), \delta_{G_2}(y)\}}
\le 2^6(2\pi+1)\Big(e^{k_{G_2}(x,y)}-1\Big)\\&\!\! \le \!\!&2^6(2\pi+1)\Big(e^{H_1}-1\Big)
= 2^6(2\pi+1)\Big(e^{H_2}-1\Big),
\end{eqnarray*}
with $H_1={\varphi_0\Big(\frac{|x-y|}{\min\{\delta_{G_2}(x), \delta_{G_2}(y)\}}\Big)}$
and $H_2={\varphi_0\Big(\frac{|x-y|}{\min\{\delta_{G}(x), \delta_{G}(y)\}}\Big)},$  which shows
that the theorem in this subcase holds with $\psi(t)=2^6(2\pi+1)\Big(e^{\varphi_0(t)}-1\Big)=:\vartheta(t).$

{\em Case B2:} There exists some positive integer $p$ such that $x,y\in B^n(x_p, \frac{5}{8}r_p)\setminus E_p.$

Let $G_3=B^n(x_p, r_p)\setminus E_p.$ Then by Theorem \ref{phiunif-E},  we know that
$G_3$ is $\varphi_0$-uniform. Hence
$$k_G(x,y)\leq k_{G_3}(x,y)\leq \varphi_0\Big(\frac{|x-y|}{\min\{\delta_{G_3}(x)
\wedge \delta_{G_3}(y)\}}\Big)
\leq \varphi_0\Big(\frac{5|x-y|}{3 \min\{\delta_G(x), \delta_{G}(y)\}}\Big).
$$

{\em Case B3:} There exists some positive integer $p$ such that $x\in B^n(x_p, \frac{1}{2}r_p)
\setminus E_p$ and \linebreak \hspace*{2.5cm} $y\in [B^n\setminus B^n(x_p,\frac{5}{8}x_p)]\setminus F$.

Choose $w\in S^{n-1}(x_p, \frac{1}{2}r_p)$ such that $d_G(w)\geq d_G(x)$. Then
$$|x-w|\leq r_p\leq 8|x-y|$$ and $$|y-w|\leq |x-w|+|x-y|\leq 9|x-y|.
$$
Hence, Case $B1$ and Case $B2$ together yield
\begin{eqnarray*}k_G(x,y)&\!\! \le \!\!& k_G(x,w)+k_G(w,y)\\&\!\!
\le \!\!&\varphi_0\Big(\frac{|x-w|}{\min\{\delta_G(x), \delta_G(w)\}}\Big)
+\vartheta(\frac{|w-y|}{\min\{\delta_G(y), \delta_G(w)\}}\Big) \\
&\!\! \le \!\! & \varphi_0\Big(\frac{8|x-y|}{\min\{\delta_G(x), \delta_G(y)\}}\Big)
+\vartheta(\frac{9|x-y|}{\min\{\delta_G(y), \delta_G(x)\}}\Big)\\
&\!\! \le \!\!&2\max\Big\{\varphi_0\Big(\frac{8|x-y|}{\min\{\delta_G(x), \delta_G(y)\}}\Big),
\vartheta\Big(\frac{9|x-y|}{\min\{\delta_G(y), \delta_G(x)\}}\Big)\Big\},
\end{eqnarray*}
which shows that in this subcase the theorem holds with
$$\psi(t)=2\max\{\varphi_0(8t),\vartheta(9t)\}=:\varphi_3(t)\,.
$$

{\em Case C:}  $x\in B^n\setminus B^n(\frac{1}{4})$, $y\in B^n(\frac{3}{16})\setminus F$.

Choose $w\in S^{n-1}(\frac{1}{4})$ such that $d_G(w)\geq d_G(y).$ Then $$\max\{|x-w|,|y-w|\}\leq 9|x-y|,$$ which shows
\begin{eqnarray*}k_G(x,y)&\!\! \le \!\!& k_G(x,w)+k_G(w,y)\\&\!\! \le \!\!&
\varphi_3\Big(\frac{|x-w|}{\min\{\delta_G(x), \delta_G(w)\}}\Big)+\varphi_2\Big(\frac{|w-y|}{\min\{\delta_G(y),
\delta_G(w)\}}\Big) \\&\!\!
\le \!\!& \varphi_3\Big(\frac{9|x-y|}{\min\{\delta_G(x), \delta_G(y)\}}\Big)
+\varphi_2\Big(\frac{9|x-y|}{\min\{\delta_G(y), \delta_G(x)\}}\Big)\\&\!\!
\le \!\!&2\max\Big\{\varphi_3\Big(\frac{9|x-y|}{\min\{\delta_G(x), \delta_G(y)\}}\Big),
\varphi_2\Big(\frac{9|x-y|}{\min\{\delta_G(y), \delta_G(x)\}}\Big)\Big\}\\&\!\!
= \!\!&\varphi_4\Big(\frac{|x-y|}{\min\{\delta_G(y), \delta_G(x)\}}\Big),
\end{eqnarray*}
where $\varphi_4(t)=2\max\{\varphi_3(9t),\varphi_2(9t)\}$.

We verified all the cases and our conclusion holds with $\psi=\varphi_4$.
\end{proof}

We now extend Theorem~\ref{paiunif3} to arbitrary domains.

\begin{theorem}\label{psiunif4} Suppose that $D\subsetneq \mathbb{R}^n$ is a $\varphi$-uniform domain and
$x_0\in D$.
Let $\{x_i\}_{i=1}^{\infty}$ be a sequence of points in $ B^n(x_0,\delta_D(x_0))$ satisfying
$x_i\in[x_0,x_{i-1})$ and $|x_0-x_i|=\frac{1}{2^{i+2}}\delta_D(x_0).$
Let $B_i$'s be  disjoint balls with centers $x_i$ and radii $r_i$ and for $c\in(0,1)$,
let $E_i\subset B^n(x_i,c r_i)$ be the closed sets whose complements with respect to $\Rn$
are $\psi$-uniform and satisfy $x_i\in E_i$.
Denote
$F=\cup E_i\cup\{x_0\}.$
Then $D\setminus F$ is $\varphi_3$-uniform with $\varphi_3$
depending on $\varphi$ and $\psi$.
\end{theorem}
\begin{proof} We note that $F\subseteq B^n(x_0, \frac{3}{16}\delta_D(x_0))$. Let $x,y\in D\setminus F$ be arbitrary.
We prove the theorem by considering three cases.

{\em Case I:} $x,y\in B^n(x_0,\frac{1}{2}\delta_D(x_0))\setminus F.$

By Theorem \ref{paiunif3} we know that $D_1=B^n(x_0,\frac{1}{2}\delta_D(x_0))\setminus F$
is $\varphi_1$-uniform with $\varphi_1$ depending only on $\varphi$ and $\psi$. Then
\begin{eqnarray*}k_{D\setminus F}(x,y)&\!\! \le \!\!&  k_{D_1}(x,y)\\&\!\! \le \!\!&
\varphi_1\Big(\frac{|x-y|}{\min\{\delta_{D_1}(x), \delta_{D_1}(y)\}}\Big)\\&\!\! = \!\!&
\varphi_1\Big(\frac{|x-y|}{\min\{\delta_{D\setminus F}(x), \delta_{D\setminus F}(y)\}}\Big).
\end{eqnarray*}

{\em Case II:} $x,y\in D_2=D\setminus B^n(x_0,\frac{1}{4}\delta_D(x_0)) .$

Since $D$ is $\varphi$-uniform, using Lemma~\ref{genVu2-lem} with $\alpha=\frac{1}{4}$ and $\theta=\frac{3}{4}$ we get
\begin{eqnarray*}k_{D\setminus F}(x,y)&\!\! \le \!\!&  k_{D_2}(x,y)\\&\!\! \le \!\!& a\Big(\frac{1}{4},
\frac{3}{4}\Big)k_D(x,y)\\&\!\! \le \!\!&  a\Big(\frac{1}{4},\frac{3}{4}\Big)\varphi\Big(\frac{|x-y|}{\min\{\delta_D(x),
\delta_D(y)\}}\Big)\\
&\!\! \le \!\!&  a\Big(\frac{1}{4},\frac{3}{4}\Big)\varphi\Big(\frac{|x-y|}{\min\{\delta_{D\setminus F}(x),
\delta_{D\setminus F}(y)\}}\Big).
\end{eqnarray*}
This gives that
$$k_{D\setminus F}(x,y)\leq \varphi_2\Big(\frac{|x-y|}{\min\{\delta_{D\setminus F}(x), \delta_{D\setminus F}(y)\}}\Big)
$$
with $\varphi_2(t)=a(\frac{1}{4},\frac{3}{4})\varphi(t).$

{\em Case III:} $x\in B^n(x_0,\frac{1}{4}\delta_D(x_0))\setminus F,$ $y\in D\setminus B^n(x_0,\frac{1}{2}\delta_D(x_0)) .$

Let $w\in S^{n-1}(x_0, \frac{1}{2}\delta_D(x_0))$. Then we have $$\delta_{D\setminus F}(w)
\geq \frac{1}{4}\delta_D(x_0)\geq \delta_{D\setminus F}(x),$$
and
$$\max\{|x-w|,|y-w|\}\leq 5|x-y|.$$
Hence
\begin{eqnarray*}k_{D\setminus F}(x,y)&\!\! \le \!\!& k_{D\setminus F}(x,w)+k_{D\setminus F}(w,y)\\&\!\!
\le \!\!&  \varphi_1\Big(\frac{|x-w|}{\min\{\delta_{D\setminus F}(x), \delta_{D\setminus F}(w)\}}\Big)\\&\!\!
\le \!\!&  2\max\Big\{ \varphi_1\Big(\frac{5|x-y|}{\min\{\delta_{D\setminus F}(x), \delta_{D\setminus F}(y)\}}\Big),
\varphi_2\Big(\frac{5|x-y|}{\min\{\delta_{D\setminus F}(x), \delta_{D\setminus E}(y)\}}\Big)\Big\}\\&\!\! \le \!\!&
\varphi_3\Big(\frac{|x-y|}{\min\{\delta_{D\setminus E}(x), \delta_{D\setminus E}(y)\}}\Big),
\end{eqnarray*}
where $\varphi_3(t)=2\max\{\varphi_1(5t),\varphi_2(5t)\}.$
Hence we complete the proof of the theorem.
\end{proof}

%} %END COMMENT
\bigskip
{\sc Acknowledgements.} This research was started in the fall of
2008 when the third author was visiting the University of Turku,
Finland, supported by CIMO, grant number TM-08-5606.
The third author also acknowledges the support of the National Board
for Higher Mathematics, DAE, India, during his post-doctoral study at IIT Madras. The work of the first author
was supported by the Graduate School of Analysis and its
Applications, Finland. The work of the second author was supported by the Academy of Finland grant of Matti Vuorinen
Project number 2600066611. and by  Hunan Provincial Innovation Foundation For Postgraduate, China.

The authors thank the referees who have made valuable comments on various versions of this manuscripts.

\end{document}